\documentclass[12pt,a4paper]{amsart}
\usepackage{amsfonts}
\usepackage{amssymb}
\usepackage{amsmath}
\usepackage{amsthm} 
\usepackage{cite}
\usepackage{enumitem}
\usepackage{tikz}
\usepackage{subfigure}
\usepackage{latexsym}
\usepackage{dsfont}

\theoremstyle{plain}
\newtheorem{thm}{Theorem} 

\newtheorem{lem}[thm]{Lemma}
\newtheorem{prop}[thm]{Proposition}
\newtheorem{cor}[thm]{Corollary}
\newtheorem{rem}[thm]{Remark}

\providecommand{\ind}{\mathds{1}} 
\providecommand{\les}{\lesssim}
\providecommand{\sm}{\setminus}
\providecommand{\N}{\mathbb{N}}
\providecommand{\R}{\mathbb{R}}

\providecommand{\C}{\mathbb{C}}
\providecommand{\eps}{\varepsilon}
\providecommand{\ov}{\overline}

\DeclareMathOperator{\supp}{supp}

\renewcommand{\qed}{\hfill $\Box$}

\setitemize{itemsep=+2pt} 
\setenumerate{itemsep=+2pt}

\begin{document}


\allowdisplaybreaks

\title[Restriction-Extension operator with symmetries]{The Restriction-Extension Operator 
on Lebesgue spaces with symmetries and applications to the Limiting Absorption Principle}

\author{Rainer Mandel}
\email{Rainer.Mandel@gmx.de}


\begin{abstract}
  We prove $L^p$-$L^q$-estimates for the Restriction-Extension operator acting on block-radial functions 
  with the aid of new oscillatory integral estimates and interpolation results in mixed Lorentz spaces. 
  We apply this to the Limiting Absorption Principle for elliptic (pseudo-)differential operators
  with constant coefficients. In this way we obtain a richer existence theory for
  Helmholtz-type problems on $\R^d$ with block-radial right hand sides.
\end{abstract}

\maketitle
\allowdisplaybreaks

\section{Introduction}

In this paper we are interested in new $L^p$-$L^q$-bounds for the Restriction-Extension operator
 $$
   Tf(x) 
   := \mathcal F^{-1}(\hat f\,d\sigma)(x)
   :=  (2\pi)^{-\frac{d}{2}} \int_{\mathbb{S}^{d-1}} \hat f(\omega) e^{ix\cdot\omega}\,d\sigma(\omega)
$$ 
of the unit sphere $\mathbb{S}^{d-1}\subset \R^d$. Here, $\mathcal F f=\hat f$ denotes the Fourier transform of
$f$ and  $\sigma$ is the canonical surface measure on the sphere. It is known that $T:L^p(\R^d)\to L^q(\R^d)$
is bounded if and only if the exponents $p,q\in [1,\infty]$ satisfy 
$$
  \min\left\{\frac{1}{p},\frac{1}{q'}\right\} > \frac{d+1}{2d},\qquad 
  \frac{1}{p}-\frac{1}{q} \geq \frac{2}{d+1}.
$$
The first condition is seen to be necessary by choosing any Schwartz function $f:\R^d\to\R$ such
that $\hat f \equiv 1$ on $\mathbb{S}^{d-1}$. Indeed, well-known properties of Bessel functions then  
imply $Tf\in L^q(\R^d)$ if and only if $\frac{1}{q'}>\frac{d+1}{2d}$, see
\eqref{eq:defJ},\eqref{eq:BesselJAsymptotics} below. Since $T$ is symmetric, $\frac{1}{p}>\frac{d+1}{2d}$
must hold, too. The necessity of the second condition   follows from the optimality of the
Stein-Tomas inequality
\begin{equation} \label{eq:STclassical}
  \int_{\mathbb{S}^{d-1}} |\hat f|^2\,d\sigma \leq C \|f\|_p^2 \qquad\text{where }f\in L^p(\R^d),\;1\leq p\leq
  \frac{2(d+1)}{d+3} 
\end{equation}
in view of the Knapp example. We shall review that argument later.
It is natural to wonder about larger ranges of exponents under more restrictive conditions on the
functions. Our aim is to analyze the effect of additional symmetry assumptions. The simplest case is given by
radial symmetry where $f(x)=f_0(|x|)$. In that case $\hat f$ is again radially symmetric on $\R^d$ and in particular constant on
$\mathbb{S}^{d-1}$. It is straightforward to show that the Stein-Tomas inequality  holds for radial
functions if and only if $\frac{1}{p}>\frac{d+1}{2d}$ and $T:L^p_{\text{rad}}(\R^d)\to L^q_{\text{rad}}(\R^d)$ is bounded if and only
if $\min\{\frac{1}{p},\frac{1}{q'}\} > \frac{d+1}{2d}$. The guiding question of this article is: what happens
between the nonsymmetric and the radially symmetric case?
 
\medskip
 
To shed some light on this issue we provide a thorough analysis for the special symmetry groups
$G_k:=O(d-k)\times O(k)\subset O(d)$ where $k\in\{1,\ldots,d-1\}$. In~\cite{ManDOS} it was shown
that the Stein-Tomas Inequality holds in the larger range
$1\leq p\leq \frac{2(d+m)}{d+m+2}$ with $m:=\min\{d-k,k\}$ provided that $f\in L^p(\R^d)$ is $G_k$-symmetric,
i.e., $f\in L_{G_k}^p(\R^d)$. So it is natural to ask whether the Restriction-Extension operator also satisfies better
bounds, which we answer affirmatively in this paper. Our first main result reads as follows:

\begin{thm} \label{thm:RestrictionExtension} 
  Assume $d\in \N,d\geq 2$ and $k\in\{1,\ldots,d-1\}$. Then $T:L^p_{G_k}(\R^d)\to L^q_{G_k}(\R^d)$ is bounded
  if and only if $p,q\in [1,\infty]$ satisfy, for  $m:=\min\{d-k,k\}$,
  \begin{equation} \label{eq:pqconditions0}
    \min\left\{\frac{1}{p},\frac{1}{q'}\right\} > \frac{d+1}{2d},\qquad 
    \frac{1}{p}-\frac{1}{q} \geq  \frac{2}{d+m}.
  \end{equation}  
\end{thm}

%
%
The improvement with respect to the general nonsymmetric situation is illustrated in the following Riesz
diagram.  
  
\begin{figure}[htbp]  
\begin{tikzpicture}[scale=12]
\draw[->] (0.5,0) -- (1.05,0) node[right]{$\frac{1}{p}$};
\draw[->] (0.5,0) -- (0.5,0.45) node[above]{$\frac{1}{q}$};
\draw (1,0.5) -- (1,0) node[below]{$1$};
\coordinate (E) at (1,0);
\coordinate (ST) at (0.66,0.33); 
\coordinate (I) at (0.625,0.292);
\coordinate (Ib) at (0.625,0);
\coordinate (I') at (0.708,0.375); 
\coordinate (I'b) at (1,0.375);
\draw [dotted] (0.5,0.375) node[left]{$\frac{d-1}{2d}$} -- (1,0.375) node[above left]{$B$};
\draw [dotted] (0.625,0) node[below]{$\frac{d+1}{2d}$} -- (0.625,0.42);
\draw (Ib) node[above left]{$B'$};
\draw (E) node[above left]{$A$};
\draw (0.5,0) node[below]{$\frac{1}{2}$};
\draw [line width = 0.5mm] (I'b) -- (I')  node[above]{$D$} -- (I)  node[left]{$D'$} -- (Ib); 
\draw [line width = 0.5mm] (E) --  (I'b) -- (0.775,0.375) node[above]{$C$} -- (0.625,0.225)
node[left]{$C'$} -- (Ib) -- (E);
\end{tikzpicture}
 \caption{Riesz diagram for the case $d=4$ and $k=m=2$: 
 The pentagon ABCC'B', with the closed segments $BC,C'B'$ excluded, contains all
 exponents such that $T:L^p(\R^d)\to L^q(\R^d)$ is bounded. 
 By Theorem~\ref{thm:RestrictionExtension} the larger pentagon ABDD'B', 
 again with the closed segments $BD,D'B'$ excluded,
 contains all exponents such that $T:L^p_{G_k}(\R^d)\to L^q_{G_k}(\R^d)$ is bounded. On
 the horizontal (resp. vertical) closed segments BC,BD (resp. C'B',D'B') the corresponding
 statements hold  with $L^q$ (resp. $L^p$) replaced by $L^{q,\infty}$ (resp. $L^{p,1}$). 
 }
 \label{fig:Riesz}
\end{figure}

\medskip

The Restriction-Extension operator for the sphere $T$ is closely related to the Limiting Absorption Principle
for the Helmholtz equation. Here the task is to find ``physical'' solutions of $-\Delta u
-  u = f$ in $\R^d$ by making sense of   
\begin{equation} \label{eq:def_uf}
  u_f := (-\Delta-1+i0)^{-1}f := \lim_{\eps\to 0^+} \mathcal F^{-1}\left(\frac{\hat
  f}{|\cdot|^2-1+i\eps}\right)
\end{equation} 
in some suitable  topology. It is   well-known   that the imaginary part
of $u_f$  equals $Tf$ up to some multiplicative constant, see for instance
\cite[Corollary~2.5]{EveqWeth2015} or formula (4.7) in~\cite{Agmon1975}. Given this it is not
surprising that the operator $(-\Delta-1+i0)^{-1}$ has similar mapping properties. 
Kenig, Ruiz and Sogge \cite{KenigRuizSogge1987} and Guti\'{e}rrez~\cite{Gutierrez2004} proved the 
$L^p(\R^d)$-$L^q(\R^d)$-boundedness
of $f\mapsto u_f$ assuming~\eqref{eq:pqconditions0} as well as $\frac{1}{p}-\frac{1}{q}\leq
\frac{2}{d}$, where the latter assumption is needed to control large frequencies. We prove the analogous statement in the
$G_k$-symmetric setting and extend the analysis to a reasonable  class of (pseudo-)differential
operators $P(|D|)$ with constant coefficients. For simplicity, we first state
the result for $P(|D|)=|D|^s-1 = (-\Delta)^{s/2}-1$ and refer to Theorem~\ref{thm:resolvent2} for a 
straightforward generalization to more general symbols.

%


\begin{thm}\label{thm:resolvent}
  Assume $d\in\N, d\geq 2, k\in\{1,\ldots,d-1\}$ and $s>0$. Then $(|D|^s-1+i0)^{-1}:L^p_{G_k}(\R^d)\to
  L^q_{G_k}(\R^d)$ is a bounded linear operator provided that $p,q\in [1,\infty]$ satisfy 
  \begin{equation} \label{eq:resolvent_conditions}
    \min\left\{\frac{1}{p},\frac{1}{q'}\right\}>\frac{d+1}{2d},\quad 
    \frac{2}{d+m}\leq  \frac{1}{p}-\frac{1}{q}\leq \frac{s}{d},\quad 
    \left(\frac{1}{p},\frac{1}{q}\right)\notin 
    \left\{ \left(1,\frac{d-s}{d}\right),\left(\frac{s}{d},0\right)\right\}.
  \end{equation}
\end{thm}

\medskip

\begin{cor} \label{cor:LAP}
  Assume $d\in\N, d\geq 2, k\in\{1,\ldots,d-1\},s>0$ and $f\in L^p_{G_k}(\R^d)$. Then the equation
  $(-\Delta)^{s/2}u-u=f$ in $\R^d$ admits the solution $u_f\in L^q_{G_k}(\R^d)$ obtained via the Limiting
  Absorption principle \eqref{eq:def_uf} provided that $p,q\in [1,\infty]$
  satisfy~\eqref{eq:resolvent_conditions}.
\end{cor}

\medskip

\begin{rem} ~
  \begin{itemize} 
    \item[(a)] Exploiting Corollary~\ref{cor:LAP} amd well-known Bessel potential
    estimates one can actually prove the more general statement $u_f\in W^{\tilde s,q}(\R^d),\tilde s\geq 0$
    whenever~\eqref{eq:resolvent_conditions} holds for $s-\tilde s$ instead of $s$.  
    \item[(b)] 
    Given the results in \cite{EveqWeth2015, ManMonPel2017} it is straightforward to
    show that Corollary~\ref{cor:LAP} allows to prove the existence of nontrivial $L^p(\R^d)$-solutions to
    the nonlinear problems of the prototypical form 
    $$
      (-\Delta)^{s/2}u  - u = \sigma |u|^{p-2}u \quad\text{in }\R^d    
    $$
    using dual variational methods. Here, $\sigma\in\R\sm\{0\}$ and, for $q:=p'=\frac{p}{p-1}$, the exponent
    $(p,q)$ lies in the interior of the set given by~\eqref{eq:resolvent_conditions}.  
  \end{itemize}
\end{rem}

\subsection{Proof idea}

We first recapitulate the proof of the optimal estimates for the Restriction-Extension operator in the
nonsymmetric setting in order to explain the difficulties that we have to overcome in our analysis. As
we recall below in Proposition~\ref{prop:AsymptoticsJ}, $T$ is a convolution operator with an explicitly known,
radially symmetric and oscillatory kernel function $\mathcal J(|\cdot|)$ that we define in~\eqref{eq:defJ}, i.e.,
$Tf=\mathcal J(|\cdot|)\ast f$. The pointwise bounds for $\mathcal J$ and resulting integrability
properties are, however, not sufficient to derive the optimal mapping properties for $T$. To take the
oscillatory nature into account, the operator is splitted dyadically according to $T= T_0 + \sum_{j=1}^\infty
T_j$ where $T_0$ is a harmless bounded linear operator of convolution type and $T_j f =
(\chi(2^{-j}|\cdot|)\mathcal J(|\cdot|))\ast f$ for $j\in\N$. The cut-off function $\chi$ is needed to localize the kernel function $\mathcal J$ inside an annulus with inner
and outer radius both comparable to $2^j$. The pointwise bounds for $\mathcal J$ and Young's Convolution
Inequality then imply 
\begin{equation}\label{eq:L1Linfty}
  \|T_j f\|_\infty \les 2^{j\frac{1-d}{2}} \|f\|_1
\end{equation} 
after one line of computations. Moreover, the Stein-Tomas Theorem yields the estimate  
$$
  \|T_j f\|_2 \les 2^{j\frac{1}{2}} \|f\|_{\frac{2(d+1)}{d+3}}.
$$ 
Applying   Bourgain's interpolation
method~\cite[Appendix]{CarberySeegerETC1999} one finds restricted weak-type estimates $\|Tf\|_{q,\infty}\les
\|f\|_{p,1}$ in the corners $B,C,C',B'$ and real interpolation theory allows to conclude. We refer to the
proof of Theorem~6 in \cite{Gutierrez2004} for the details.

\medskip

In the $G_k$-symmetric case new difficulties arise. Firstly, we  have to replace the bound induced the
classical Stein-Tomas Theorem by the corresponding $G_k$-symmetric version 
\begin{equation} \label{eq:strategyII}
  \|T_j f\|_2 \les 2^{j\frac{1}{2}} \|f\|_{\frac{2(d+m)}{d+m+2}},\quad
  m=\min\{k,d-k\}
\end{equation} 
for all $G_k$-symmetric functions in a rather straightforward manner. This relies on the
$G_k$-symmetric Stein-Tomas Theorem from~\cite{ManDOS} that we recall in~\eqref{eq:STnew} below. So this
crucial and non-trivial part of the proof may essentially be taken from the literature.
The main difficulty is then to prove a counterpart of \eqref{eq:L1Linfty} that leads to an optimal result in
the $G_k$-symmetric setting.
In Corollary~\ref{cor:Xpestimates} we will prove a Lorentz-space estimate
\begin{equation} \label{eq:strategyI}
  \|T_j f\|_{\frac{2m}{m-1},\infty} \les 2^{j\frac{m-d}{2m}} \|f\|_{\frac{2m}{m+1},1} \qquad 
  m\neq \frac{d}{2}.
\end{equation} 
The proof relies on a passage to polar coordinates (with two radial coordinates) and delicate pointwise
bounds for the kernel function in these coordinates, see~Proposition~\ref{prop:KjBound}. This
is based on oscillatory integral estimates that we defer to the Appendix
(Theorem~\ref{thm:OscillatoryIntegral}) due to their purely technical nature.
We shall have to switch to the more complicated setting of mixed Lorentz spaces in order to prove a 
counterpart of~\eqref{eq:strategyI} in the exceptional case $m=\frac{d}{2}$. 
Roughly speaking, this configuration is special and requires a separate analysis because  
$\frac{2m}{m+1}$ becomes an endpoint exponent  with respect to both the
$k$-dimensional and the $(d-k)-$dimensional variable,
see Lemma~\ref{lem:XpEstimates} for details.  
The final step is, in the case $m\neq \frac{d}{2}$, to combine the  estimates
\eqref{eq:strategyII},\eqref{eq:strategyI} via interpolation using   Bourgain's summation
method. This works out nicely in the case $m\neq \frac{d}{2}$, but 
an analogous interpolation scheme for $m=\frac{d}{2}$ requires more work given that 
real interpolation of estimates between mixed Lorentz spaces is needed. To solve this subtle
problem we use recent results from real interpolation theory of mixed Lorentz
spaces~\cite[Corollary~1]{Mandel2023}. In this way we derive the $L^{p,1}-L^{q,\infty}$-estimates at the
points D' and D in Figure~\ref{fig:Riesz} that lie on the diagonal line $\frac{1}{p}-\frac{1}{q}=\frac{2}{d+m}$.
 The proof is then completed via interpolation with the easier bounds in the corners B',B,A.

\subsection{Notation}

In the following let always $d\in\N,d\geq 2$ denote the space dimension. We denote by 
$L^p_{G_k}(\R^d)$ the Banach space consisting of $G_k$-symmetric complex-valued functions that belong to
$L^p(\R^d)$ where $G_k=O(d-k)\times O(k)$. Accordingly, $\mathcal S_{G_k}(\R^d)$ consists of $G_k$-symmetric Schwartz functions. 
Functions $f\in L^p_{G_k}(\R^d)$ are block-radial and admit a profile function
$f_0$ such that $f(x)=f_0(|y|,|z|)$ for $x=(y,z)\in\R^{d-k}\times\R^k$. The norm is then given by   
$$
  \|f\|_p^p
  := \|f\|_{L^p(\R^d)}^p = c_{d,k} \int_0^\infty \rho_1^{d-k-1}\rho_2^{k-1}
  |f_0(\rho_1,\rho_2)|^p\,d\rho_1\,d\rho_2
$$
where $c_{d,k}:= |\mathbb{S}^{d-k-1}||\mathbb{S}^{k-1}|>0$. Here, $\mathbb{S}^{l-1}:=
\{z\in\R^l: |z|=1\}$ denotes the unit sphere in $\R^l$ with $(l-1)$-dimensional Hausdorff measure
$|\mathbb{S}^{l-1}|$, $\sigma$ represents the canonical surface measure on this sphere. The dependence on the
dimension will be clear from the context. We shall often write $m:= \min\{k,d-k\}$ and our focus lies
on the case $m\geq 2$ where our results truly improve the known ones. We will need the
 $G_k$-symmetric improvement of the Stein-Tomas inequality
\begin{equation} \label{eq:STnew}
  \int_{\mathbb S^{d-1}} |\hat f|^2\,d\sigma 
  \les \|f\|_{p}^2 
  \quad\text{for all }f\in L^{p}_{G_k}(\R^d) \text{ and } 1\leq p\leq p_{ST} := \frac{2(d+m)}{d+m+2} 
\end{equation}
from~\cite{ManDOS}.
Note that the $G_k$-symmetric threshold exponent $p_{ST}$ is larger than the 
usual threshold exponent $\frac{2(d+1)}{d+3}$ for the corresponding inequality without any symmetry
constraint, see~\eqref{eq:STclassical}.
Here and in the following, the symbol $\les$ stands for $\leq C$ for some positive number $C$ only depending
on the fixed parameters like $d,m,p$.  The Fourier transform 
$$
  \hat f(\xi) := \mathcal F f(\xi) := (2\pi)^{-\frac{d}{2}} \int_{\R^d} f(x) e^{-ix\cdot\xi}\,dx 
$$
preserves $G_k$-symmetry, and $f\in\mathcal S_{G_k}(\R^d)$ holds if and only if $\hat f \in\mathcal
S_{G_k}(\R^d)$.
We will write $\phi(|D|)u:= \mathcal F^{-1}(\phi(|\cdot|)\hat u)$. Define
\begin{align} \label{eq:defJ}
   \mathcal J(|x|)
   := \mathcal F^{-1}(1\,d\sigma)(x)
   := (2\pi)^{-\frac{d}{2}} \int_{\mathbb{S}^{d-1}} e^{-ix\cdot \omega}\,d\sigma(\omega)
   = c_d |x|^{\frac{2-d}{2}}J_{\frac{d-2}{2}}(|x|)
 \end{align}
where $c_d>0$ is a suitable dimensional constant. Known asymptotic expansions of the Bessel
functions of the first kind \cite[p.356]{Stein1993} imply, for any given $L\in\N_0$, 
\begin{equation}\label{eq:BesselJAsymptotics}
    \mathcal J(z) 
    = \sum_{l=0}^{L-1} |z|^{\frac{1-d}{2}-l}(\alpha_l e^{i|z|} + \ov{\alpha_l}e^{-i|z|}) 
     + O(|z|^{\frac{1-d}{2}-L})\quad\text{as }|z|\to\infty
\end{equation}
for some $\alpha_0,\ldots,\alpha_{L-1}\in \C\sm\{0\}$. We shall also use the decomposition  
  \begin{equation} \label{eq:decomposition}
    \mathcal J(s) = \mathcal J^1(s)+s^{\frac{1-d}{2}}\mathcal J^2(s)e^{is}+s^{\frac{1-d}{2}}\ov{\mathcal J^2(s)}e^{-is}
  \end{equation}
where $\mathcal J^1$ is smooth with compact support near the origin and $\mathcal J^2$ is smooth with
unbounded support not containing the origin such that $|(\mathcal J^2)^{(k)}(s)|\les s^{-k}$ as $|s|\to\infty$
for all $k\in\N_0$. In view of \eqref{eq:BesselJAsymptotics} we even have 
$\mathcal J^2(s)= \sum_{l=0}^{L-1} \alpha_l s^{-l} + O(|s|^{-L})$ as $|s|\to\infty$. For a proof
of~\eqref{eq:decomposition} see \cite[Proposition~6]{Mandel2020}.

\section{The $G_k$-symmetric Restriction-Extension operator for the sphere}
 
In this section we prove the boundedness of the Restriction-Extension operator $T:L^p_{G_k}(\R^d)\to
L^q_{G_k}(\R^d)$ for all $p,q\in [1,\infty]$ satisfying \eqref{eq:pqconditions0}. So the   goal is to
prove Theorem~\ref{thm:resolvent}. The general idea of the proof is similar to the one in the non-symmetric
case, but the required tools require new methods in Fourier Restriction Theory, Oscillatory Integral Theory and Real Interpolation Theory
for mixed Lorentz spaces that we developed in~\cite{Mandel2023}. 
We will focus on the case $m=\min\{k,d-k\}\geq 2$ given that the result for   
$m=1$ is   covered by~\cite[Theorem~6]{Gutierrez2004}. 
We start with a representation formula for the Restriction-Extension operator. It will be convenient to fix 
a smooth function $\tau:\R\to\R$ such that $\tau(r)=1$ for $r\in [\frac{3}{4},\frac{5}{4}]$  and
$\supp(\tau)\subset [\frac{1}{2},\frac{3}{2}]$.

\begin{prop} \label{prop:AsymptoticsJ}
  We have $Tf = \mathcal J(|\cdot|) \ast (\tau(|D|)f)$ for all $f\in\mathcal S_{G_k}(\R^d)$. 
\end{prop}
\begin{proof}
  The   claim follows from   $\tau(1)=1$ and 
  $$
    \mathcal F^{-1}(\hat f\,d\sigma) 
    = \mathcal F^{-1}(\tau(|\cdot|)\hat f\,d\sigma)
    = \mathcal F^{-1}(1\,d\sigma)\ast (\tau(|D|)f) 
    \stackrel{\eqref{eq:defJ}}{=}  \mathcal J(|\cdot|)\ast (\tau(|D|)f).
  $$ 
\end{proof}

Next we exploit the asymptotic expansion of $\mathcal J$ at infinity from \eqref{eq:BesselJAsymptotics} in
order to split up the kernel function $\mathcal J(|\cdot|)$.
The parameter $L>\frac{d-1}{2}$ will remain fixed, so the remainder term in \eqref{eq:BesselJAsymptotics} is
bounded and integrable and its mapping properties are easily determined. The much more difficult task is to
uncover the optimal mapping properties of the slowly-decaying oscillatory parts of $\mathcal J$ where simple pointwise estimates are not sufficient. To achieve this we choose
cut-off functions $\chi,\chi_0\in C_0^\infty(\R)$ that satisfy 
\begin{equation} \label{eq:defChi}
  \supp(\chi)\subset [\frac{1}{2},2],\qquad
  \chi_0(z) + \sum_{j\geq 1} \chi(2^{-j}z)=1 \text{ on }[0,\infty],
\end{equation}
see~\cite[Lemma~6.1.7]{BerghLoefstroem1976}. So we have $Tf= T_0f + \sum_{j=1}^\infty T_j f$ where 
$T_0:L^p_{G_k}(\R^d)\to L^q_{G_k}(\R^d)$ is bounded for $1\leq p\leq q\leq \infty$ and, for $j\geq
1$,
\begin{equation}\label{eq:defPhij}
  T_jf := \Phi_j\ast (\tau(|D|)f)
  \quad\text{where }      
    \Phi_j(z) :=  \chi(2^{-j}|z|)\cdot  \sum_{l=0}^{L-1} 
    |z|^{\frac{1-d}{2}-l}(\alpha_l e^{i|z|} + \ov{\alpha_l}e^{-i|z|}). 
\end{equation}
Fourier restriction theory for $G_k$-symmetric functions from~\cite{ManDOS} gives the following.
   
\begin{lem}\label{lem:FourierEstimate}
  For all $f\in\mathcal S_{G_k}(\R^d)$  and $j\geq 1$
  $$
    \|T_j f\|_2 \les 2^{j\frac{1}{2}} \|f\|_{p_{ST}},\qquad
    \|T_j f\|_{p_{ST}'} \les 2^{j\frac{1}{2}} \|f\|_2. 
  $$
\end{lem}
\begin{proof} 
  We mimick the proof of~\cite[Theorem~6]{Gutierrez2004} where the corresponding result was shown in the
  nonsymmetric case. Writing $\varphi_j(|\xi|):=\hat \Phi_j(\xi)$ we deduce from Plancherel's identity and the
  $G_k$-symmetric Stein-Tomas Theorem~\eqref{eq:STnew}
 \begin{align*}
    \|T_jf\|_2^2
    &=  \|\hat \Phi_j \tau(|\cdot|) \hat  f\|_2^2 
    \;=\; \int_0^\infty \varphi_j(r)^2 \tau(r)^2 r^{d-1} \left( \int_{\mathbb S^{d-1}} |\hat
    f(r\omega)|^2\,d\sigma(\omega) \right)\,dr  \\
    &\les \int_{\frac{1}{2}}^{\frac{3}{2}} \varphi_j(r)^2 r^{d-1}  \|f\|_{p_{ST}}^2 \,dr 
    \les \|\Phi_j\|_2^2 \|f\|_{p_{ST}}^2 
    \;\les\; 2^j \|f\|_{p_{ST}}^2.
  \end{align*}
  Here we used $\supp(\tau)\subset [\frac{1}{2},\frac{3}{2}]$ and, in the last estimate, 
  $$
    |\Phi_j(z)|
    \les \chi(2^{-j}|z|)|z|^{\frac{1-d}{2}}
    \les 2^{j\frac{1-d}{2}} \ind_{2^{j-1}\leq |z|\leq 2^{j+1}}
    \qquad\text{for all }z\in\R^d.
  $$ 
  Since $T_j$ is selfadjoint, this implies both inequalities.
\end{proof}
 
We need another bound as a replacement for the $L^1$-$L^\infty$-estimate in the nonsymmetric setting.
It relies pointwise estimates for the kernel function that we will prove using oscillatory integral
estimates. The lengthy proof of the latter is deferred to the Appendix
(Theorem~\ref{thm:OscillatoryIntegral}).
We introduce the phase function 
\begin{equation} \label{eq:def_Psij}
  \Psi_j(s_1,s_2):=2^{-j}\sqrt{t_1^2+t_2^2+\rho_1^2+\rho_2^2-2t_1\rho_1s_1-2t_2\rho_2s_2} 
  \qquad \text{for } s_1,s_2\in [-1,1]
\end{equation}   
and the first step is to rewrite the convolution $T_jf = \Phi_j\ast (\tau(|D|)f)$ in polar coordinates. We
recall our notational convention  $x=(y,z)\in\R^{d-k}\times\R^k$.

\begin{prop}  \label{prop:RepresentationFormula}
   We have for all $f\in\mathcal S_{G_k}(\R^d)$ and for all $j\geq 1$ 
  \begin{equation}\label{eq:Tjf_formula}
    T_jf(x) = \int_0^\infty \int_0^\infty K_j(t_1,t_2,\rho_1,\rho_2)
    \rho_1^{d-k-1}\rho_2^{k-1} f_0(\rho_1,\rho_2) \,d\rho_1\,d\rho_2
  \end{equation}
  where $|y|=t_1,|z|=t_2$, $f_0$ denotes the block-radial profile of $\tau(|D|)f$ and 
  $$
    K_j(t_1,t_2,\rho_1,\rho_2)
    = |\mathbb{S}^{d-k-2}||\mathbb{S}^{k-2}| \int_{-1}^1 \int_{-1}^1 (1-s_1^2)^{\frac{d-k-3}{2}}
    (1-s_2^2)^{\frac{k-3}{2}} \Phi_j(2^j \Psi_j(s_1,s_2)) \,ds_1\,ds_2.
  $$
\end{prop}
\begin{proof}
  Passing to polar coordinates in $\R^{d-k}$ and $\R^k$, respectively, we find   \eqref{eq:Tjf_formula}
  for 
  $$
    K_j(|y|,|z|,\rho_1,\rho_2)  
     =  \int_{\mathbb S^{d-k-1}} \int_{\mathbb S^{k-1}}
    \Phi_j(|x-(\rho_1\omega_1,\rho_2\omega_2)|) \,d\sigma(\omega_1)\,d\sigma(\omega_2). 
  $$
  The Funk-Hecke formula \cite[p.30]{Mueller1998} gives
  \begin{align*}
    &\; \int_{\mathbb S^{d-k-1}} \int_{\mathbb S^{k-1}}
    \Phi_j(|x-(\rho_1\omega_1,\rho_2\omega_2)|) \,d\sigma(\omega_1)\,d\sigma(\omega_2) \\
    &= \int_{\mathbb S^{d-k-1}} \int_{\mathbb S^{k-1}}
    \Phi_j\left(\sqrt{|y|^2+|z|^2+\rho_1^2+\rho_2^2-2\rho_1
    y\cdot\omega_1-2\rho_2z\cdot\omega_2}\right) \,d\sigma(\omega_1)\,d\sigma(\omega_2) \\
    &=  |\mathbb{S}^{d-k-2}||\mathbb{S}^{k-2}|  \int_{-1}^1 \int_{-1}^1 
     (1-s_1^2)^{\frac{d-k-3}{2}} (1-s_2^2)^{\frac{k-3}{2}}
    \Phi_j(2^j\Psi_j(s_1,s_2)) \,ds_1\,ds_2,  
  \end{align*}
  which is all we had to show.
\end{proof}
 
 Our estimate for the kernel function in polar coordinates reads as follows: 

\begin{prop} \label{prop:KjBound}
    We have for all $j\geq 1$ and $t_1,t_2,\rho_1,\rho_2\geq 0$ 
  \begin{equation} \label{eq:est_Kj}
    |K_j(t_1,t_2,\rho_1,\rho_2)|
    \les 2^{j\frac{1-d}{2}}
    \min\left\{ 1, (2^{-j}\rho_1t_1)^{-\frac{d-k-1}{2}}\right\}
    \min\left\{ 1, (2^{-j}\rho_2t_2)^{-\frac{k-1}{2}}\right\}.
  \end{equation}
\end{prop} 
\begin{proof}
   Combining the formulas for $K_j$ from Proposition~\ref{prop:RepresentationFormula} and
   \eqref{eq:defPhij} gives   
   \begin{align*}
     &\; |K_j(t_1,t_2,\rho_1,\rho_2)| \\  
    &\les  \sum_{l=0}^{L-1}   2^{j(\frac{1-d}{2}-l)}  \left|\int_{-1}^1\int_{-1}^1
    (1-s_1^2)^{\frac{d-k-3}{2}} (1-s_2^2)^{\frac{k-3}{2}} \chi_l(\Psi_j(s_1,s_2)) e^{i \lambda\Psi_j(s_1,s_2)}
    \,ds_1\,ds_2\right|
  \end{align*}
  with $\lambda=2^j$ and $\chi_l(z):= \chi(z)|z|^{\frac{1-d}{2}-l}$, so  $\chi_l\in C_0^\infty(\R)$
  with $\supp(\chi_l)=\supp(\chi)\subset[\frac{1}{2},2]$. The phase function from~\eqref{eq:def_Psij} can be
  written as $\Psi_j(s)=\sqrt{A-B_1s_1-B_2s_2}$ where 
  $$
    A= 4^{-j}(\rho_1^2+\rho_2^2+t_1^2+t_2^2),\quad
    B_1 = 4^{-j}\cdot 2\rho_1t_1,\quad B_2 = 4^{-j}\cdot
   2\rho_2t_2.
  $$   
  In Theorem~\ref{thm:OscillatoryIntegral} (see Appendix) we prove the following estimate:
  \begin{align} \label{eq:oscint_estimate}
    \begin{aligned}
     &\;\left|\int_{-1}^1\int_{-1}^1
    (1-s_1^2)^{\frac{d-k-3}{2}} (1-s_2^2)^{\frac{k-3}{2}} \chi_l(\Psi_j(s_1,s_2)) e^{i \lambda\Psi_j(s_1,s_2)}
    \,ds_1\,ds_2\right| \\
    &\les  \min\left\{ 1, |\lambda B_1|^{-\frac{d-k-1}{2}}\right\}  \min\left\{ 1,
    |\lambda B_2|^{-\frac{k-1}{2}}\right\}. 
  \end{aligned}
  \end{align}
  Plugging in the values for $B_1,B_2,\lambda$ we find
  \begin{align*}
    |K_j(t_1,t_2,\rho_1,\rho_2)|
    &\les \sum_{l=0}^{L-1}   2^{j(\frac{1-d}{2}-l)}\cdot \min\left\{ 1, (2^{-j}\rho_1t_1)^{-\frac{d-k-1}{2}}\right\}
    \min\left\{ 1, (2^{-j}\rho_2t_2)^{-\frac{k-1}{2}}\right\} \\
    &\les 2^{j\frac{1-d}{2}} \min\left\{ 1, (2^{-j}\rho_1t_1)^{-\frac{d-k-1}{2}}\right\}
    \min\left\{ 1, (2^{-j}\rho_2t_2)^{-\frac{k-1}{2}}\right\}.
  \end{align*}
\end{proof}

\begin{rem}
  We emphasize that the presence of the oscillatory factor
  $e^{i\lambda\Psi_j(s_1,s_2)}$ is crucial for our application. In fact, 
  the pointwise bound for $K_j$ cannot be proved without it. Indeed, for $t_1=\rho_1=2^j$ and $t_2=\rho_2=0$
  the term $\Psi_j(s_1,s_2)$ is independent of $j$. So the integral
  $$
    \int_{-1}^1\int_{-1}^1
    (1-s_1^2)^{\frac{d-k-3}{2}} (1-s_2^2)^{\frac{k-3}{2}} \chi_l(\Psi_j(s_1,s_2))\,ds_1\,ds_2
  $$
  is constant with respect to $j$ whereas the upper bound \eqref{eq:est_Kj} decays to zero. This decay is due
  to the oscillatory factor.
\end{rem} 

These pointwise bounds for $K_j$ reveal a different behaviour with respect to the $y$- and $z$-variable. 
To take this into account we consider mixed norm spaces and introduce
\begin{align*}
  \mathcal L^{p_1}_y &:= L^{p_1,1}(\R^{d-k})\quad\text{if }p_1=\frac{2(d-k)}{d-k+1}
  \quad\text{and}\quad 
  \mathcal L^{p_1}_y := L^{p_1}(\R^{d-k})\quad\text{if }1\leq p_1<\frac{2(d-k)}{d-k+1}, \\
  \mathcal L^{p_2}_z &:= L^{p_2,1}(\R^k)\qquad\text{if } p_2=\frac{2k}{k+1}
  \;\;\,\quad\quad\text{and}\quad    
  \mathcal L^{p_2}_z := L^{p_2}(\R^k)\qquad \text{if } 1\leq p_2<\frac{2k}{k+1}.
\end{align*}
Recall  $\tau(|D|)f(x) = f_0(|y|,|z|)$ with $x=(y,z)$
and $y\in\R^{d-k},z\in\R^k$. It turns out that our estimates can be nicely formulated in the Banach spaces  
$$
  X_{\vec p}:= \mathcal L^{p_1}_y(\mathcal L^{p_2}_z) +  \mathcal L^{p_2}_z(\mathcal L^{p_1}_y)
  \qquad\text{where }\vec p:=(p_1,p_2).
$$
The corresponding norm is given by 
\begin{align*}
  \|u\|_{X_{\vec p}} 
  &:= \inf_{u_1+u_2=u} \|u_1\|_{\mathcal L^{p_1}_y(\mathcal L^{p_2}_z)} 
  + \|u_2\|_{\mathcal L^{p_2}_z(\mathcal L^{p_1}_y)}.  
  \\
  &= \inf_{u_1+u_2=u} \Big\| \|u_1(y,z)\|_{\mathcal L^{p_2}_z} \Big\|_{\mathcal L_y^{p_1}}
   +  \Big\| \|u_2(y,z)\|_{\mathcal L^{p_1}_y} \Big\|_{\mathcal L_z^{p_2}}.
\intertext{ 
Its dual is, thanks to $p_1,p_2<\infty$, $X_{\vec p}' = (\mathcal L^{p_1}_y)' (\mathcal L^{p_2}_z)' \cap
(\mathcal L^{p_2}_z)'(\mathcal L^{p_1}_y)'$, see~\cite[Proposition~1]{Mandel2023}. The corresponding norm is
} 
  \|u\|_{X_{\vec p}'} 
  &=  \|u \|_{(\mathcal L^{p_1}_y)'(\mathcal L^{p_2}_z)'} 
  + \|u \|_{(\mathcal L^{p_2}_z)'(\mathcal L^{p_1}_y)'}.
\end{align*}
Depending on $p_1,p_2$, this space may 
be rewritten in terms of Lebesgue spaces $L^{p_1'}(\R^{d-k}),L^{p_2'}(\R^{k})$ or 
Lorentz spaces $L^{p_1',\infty}(\R^{d-k}),L^{p_2',\infty}(\R^k)$.
Note that in most cases we have $\mathcal L^{p_1}_y(\mathcal L^{p_2}_z) \neq  \mathcal
L^{p_2}_z(\mathcal L^{p_1}_y)$, see \cite[p.302]{BenedekPanzone1961}. 
The exceptional case is $p_1=p_2=: r\in [1,\frac{2m}{m+1})$ where 
$\mathcal L^r_y(\mathcal L^r_z)=\mathcal L^r_z(\mathcal L^r_y)=L^r(\R^d)$ by the Tonelli-Fubini Theorem.

\begin{lem}\label{lem:XpEstimates}
  Assume $1\leq p_1\leq \frac{2(d-k)}{d-k+1}$ and $1\leq p_2\leq \frac{2k}{k+1}$. Then we have for all
  $f\in\mathcal S_{G_k}(\R^d)$ 
  $$
    \|T_j f\|_{X_{\vec p}'} \les 2^{j(\frac{1+d}{2}-\frac{d-k}{p_1}-\frac{k}{p_2})}
    \|f\|_{X_{\vec p}}. 
  $$   
\end{lem}
\begin{proof}
   Proposition~\ref{prop:RepresentationFormula} implies for $x=(y,z),|y|=t_1,|z|=t_2$
   \begin{align*}
     T_jf(x) 
     &= \int_0^\infty \int_0^\infty  K_j(t_1,t_2,\rho_1,\rho_2)
     \rho_1^{d-k-1}\rho_2^{k-1} f_0(\rho_1,\rho_2)\,d\rho_2 \,d\rho_1. 
   \end{align*} 
   We now use the pointwise bounds for $K_j$ from Proposition~\ref{prop:KjBound} and H\"older's
   inequality in Lorentz spaces with respect to the $\rho_2$-variable. So we get for $p_2=\frac{2k}{k+1}$ the
   estimate
   \begin{align*}
     |T_j f(x)| 
     &\les \int_0^\infty  \|
     K_j(t_1,t_2,\rho_1,\cdot) \rho_1^{d-k-1} (\cdot)^{\frac{k-1}{p_2'}} \|_{L^{p_2',\infty}(\R_+)}
     \|(\cdot)^{\frac{k-1}{p_2}}f_0(\rho_1,\cdot)\|_{L^{p_2,1}(\R_+)}\,d\rho_1 \\
     &\stackrel{\eqref{eq:est_Kj}}\les 2^{j\frac{1-d}{2}} \int_0^\infty \rho_1^{d-k-1}
      \min\left\{ 1, (2^{-j}\rho_1t_1)^{-\frac{d-k-1}{2}}\right\}  \cdot \\
      &\qquad\qquad\qquad
    \left\|\min\left\{ 1,
    (2^{-j}(\cdot)t_2)^{-\frac{k-1}{2}}\right\} (\cdot)^{\frac{k-1}{p_2'}}
    \right\|_{L^{p_2',\infty}(\R_+)}\|f_0(\rho_1,|\cdot|)\|_{L^{p_2,1}(\R^k)}\,d \rho_1  \\
     &\les 2^{j\frac{1-d}{2}} (2^jt_2^{-1})^{\frac{k}{p_2'}} \int_0^\infty 
     \rho_1^{d-k-1}  \min\{ 1, (2^{-j}\rho_1t_1)^{-\frac{d-k-1}{2}}\}
     \|f_0(\rho_1,|\cdot|)\|_{L^{p_2,1}(\R^k)} \,d\rho_1. 
   \end{align*}   
   In the case $1\leq p_2<\frac{2k}{k+1}$ one may use the classical H\"older inequality 
   instead. We conclude for $1\leq p_2\leq \frac{2k}{k+1}$ 
   \begin{align*}
     |T_j f(x)| 
     \les 2^{j\frac{1-d}{2}} (2^jt_2^{-1})^{\frac{k}{p_2'}} \int_0^\infty  \rho_1^{d-k-1}
      \min\{ 1,(2^{-j}\rho_1t_1)^{-\frac{d-k-1}{2}} \}
     \|f_0(\rho_1,|z|)\|_{\mathcal L^{p_2}_z} \,d\rho_1. 
   \end{align*}
   The analogous estimate with respect to the $y$-variable gives
   \begin{align} \label{eq:Tj_estimate}
     \begin{aligned}
     |T_j f(x)| 
     &\les 2^{j\frac{1-d}{2}} (2^jt_2^{-1})^{\frac{k}{p_2'}} (2^jt_1^{-1})^{\frac{d-k}{p_1'}}
     \Big\| \|f_0(|y|,|z|)\|_{\mathcal L^{p_2}_z} \Big\|_{\mathcal L^{p_1}_y} \\
     &= 2^{j\frac{1+d}{2}-\frac{d-k}{p_1}-\frac{k}{p_2}} |y|^{-\frac{d-k}{p_1'}} |z|^{-\frac{k}{p_2'}}
     \|\tau(|D|) f\|_{\mathcal L^{p_1}_y(\mathcal L^{p_2}_z)} \\
     &\les 2^{j\frac{1+d}{2}-\frac{d-k}{p_1}-\frac{k}{p_2}} |y|^{-\frac{d-k}{p_1'}} |z|^{-\frac{k}{p_2'}}
     \|f\|_{\mathcal L^{p_1}_y(\mathcal L^{p_2}_z)}.  
     \end{aligned}
   \end{align}
   In the last estimate we applied Young's convolution inequality in mixed norm Lorentz spaces
   from \cite[Theorem~2]{KhaiTri2014} to $\tau(|D|)f = K\ast f$ where $K:=\mathcal F^{-1}(\tau(|\cdot|))$ is a
   Schwartz function. (The result for mixed Lebesgue spaces can be found in 
   \cite[Theorem II.1.(b)]{BenedekPanzone1961} or \cite[Theorem~3.1]{GreySinnamon2016}.)
   
   \medskip
   
   Considering the estimate~\eqref{eq:Tj_estimate} for a fixed pair $(p_1,p_2)$ is not enough to prove our
   claim, but real interpolation theory does the job unless we are in an endpoint case. To see this, set the
   first endpoint to be 1 and the second endpoint to be $r:=\frac{2(d-k)}{d-k+1}$. Our first aim is to
   deduce, for any fixed $z\in\R^k$ and $p_1$ between these endpoints,
   \begin{align} \label{eq:Tj_estimateII} 
     \|T_j f(\cdot,z)\|_{(\mathcal L^{p_1}_y)'}     
     &\les 2^{j\frac{1+d}{2}-\frac{d-k}{p_1}-\frac{k}{p_2}}  |z|^{-\frac{k}{p_2'}}
     \|f\|_{\mathcal L^{p_1}_y(\mathcal L^{p_2}_z)}. 
   \end{align}   
   For the first endpoint $p_1=1$ we have $(\mathcal L^{p_1}_y)'=L^\infty(\R^{d-k})$ and 
   for the second endpoint case $p_1= r$ we have $(\mathcal L^{p_1}_y)'=L^{r',\infty}(\R^{d-k})$. In these
   cases, estimate~\eqref{eq:Tj_estimateII} is immediate from~\eqref{eq:Tj_estimate}. For $1<p_1<r$ we use
   real interpolation  and choose $\theta\in (0,1)$ such that
   $\frac{1-\theta}{1}+\frac{\theta}{r}=\frac{1}{p_1}$. 
   Then $\mathcal L_y^{p_1}=L^{p_1}(\R^{d-k})$ and real interpolation of the
   estimates in~\eqref{eq:Tj_estimate} gives
   \begin{align*}
     \|T_j f(\cdot,z)\|_{(\mathcal L^{p_1}_y)'}
     &= \|T_j f(\cdot,z)\|_{L^{p_1'}(\R^{d-k})} \\
     &\eqsim \|T_j f(\cdot,z)\|_{(L^{\infty}(\R^{d-k}),L^{r',\infty}(\R^{d-k}))_{\theta,p_1'}} \\
     &\stackrel{\eqref{eq:Tj_estimate}}\les 
     2^{j\frac{1+d}{2}-(d-k)(\frac{1-\theta}{1}+\frac{\theta}{r})-\frac{k}{p_2}}  
     |z|^{-\frac{k}{p_2'}} \|f\|_{(L^1(\R^{d-k})(\mathcal L^{p_2}_z),L^{r,1}(\R^{d-k})(\mathcal
     L^{p_2}_z))_{\theta,p_1'}}  \\     
     &\les 
     2^{j\frac{1+d}{2}-\frac{d-k}{p_1}-\frac{k}{p_2}}  
     |z|^{-\frac{k}{p_2'}} \|f\|_{(L^1(\R^{d-k}),L^{r,1}(\R^{d-k}))_{\theta,p_1'}(\mathcal L^{p_2}_z)} \\
     &\eqsim 
     2^{j\frac{1+d}{2}-\frac{d-k}{p_1}-\frac{k}{p_2}}  
     |z|^{-\frac{k}{p_2'}} \|f\|_{L^{p_1,p_1'}(\R^{d-k})(\mathcal L^{p_2}_z)} \\
     &\les 2^{j\frac{1+d}{2}-\frac{d-k}{p_1}-\frac{k}{p_2}}   |z|^{-\frac{k}{p_2'}}
     \|f\|_{L^{p_1}(\R^{d-k})(\mathcal L^{p_2}_z)}\\
     &= 2^{j\frac{1+d}{2}-\frac{d-k}{p_1}-\frac{k}{p_2}}   |z|^{-\frac{k}{p_2'}}
     \|f\|_{\mathcal L^{p_1}_y(\mathcal L^{p_2}_z)}.
   \end{align*} 
   In the fourth line we used Corollary~4.5 in \cite{ChenSun2022} and in the sixth line we used $p_1\leq
   p_1'$ due to $1\leq p_1<r<2$.
   This proves~\eqref{eq:Tj_estimateII}.
   
   \medskip
   
   Now we perform the analogous argument with respect to $z$. The endpoint cases are now 
   $p_2=1$ and $p_2=r:=\frac{2k}{k+1}$. Again, for $p_2=1$ and $p_2=r$ the estimate
    \begin{align*}
     \|T_j f \|_{(\mathcal L^{p_2}_z)'(\mathcal L^{p_1}_y)'}     
     &\les 2^{j\frac{1+d}{2}-\frac{d-k}{p_1}-\frac{k}{p_2}}  
     \|f\|_{\mathcal L^{p_1}_y(\mathcal L^{p_2}_z)}  
   \end{align*}    
   is an immediate consequence of~\eqref{eq:Tj_estimateII}. So it remains to 
   consider the case $1<p_2<r$ where we use real interpolation once more. 
   We choose $\theta\in (0,1)$ such that $\frac{1-\theta}{1}+\frac{\theta}{r}=\frac{1}{p_2}$. Then 
   \begin{align} \label{eq:InterpolatedBoundsI}
     \begin{aligned}
     \|T_j f\|_{(\mathcal L^{p_2}_z)'(\mathcal L^{p_1}_y)'}
     &= \|T_j f\|_{L^{p_2'}(\R^k)(\mathcal L^{p_1}_y)'} \\
     &\les   \|T_j f\|_{(L^{\infty}(\R^k)(\mathcal L^{p_1}_y)',L^{r',\infty}(\R^k)(\mathcal
     L^{p_1}_y)')_{\theta,p_2'}}  \\      
     &\stackrel{\eqref{eq:Tj_estimateII}}\les
       2^{j\frac{1+d}{2}-\frac{d-k}{p_1}-k(\frac{1-\theta}{1}+\frac{\theta}{r})}    
     \|f\|_{(\mathcal L^{p_1}_y(L^1(\R^k)),\mathcal L^{p_1}_y(L^{r,1}(\R^k)))_{\theta,p_2'}}
     \\
     &\les  2^{j\frac{1+d}{2}-\frac{d-k}{p_1}-\frac{k}{p_2}}    
     \|f\|_{\mathcal L^{p_1}_y((L^1(\R^k),L^{r,1}(\R^k))_{\theta,p_2'})} \\
     &\eqsim  2^{j\frac{1+d}{2}-\frac{d-k}{p_1}-\frac{k}{p_2}}    
     \|f\|_{\mathcal L^{p_1}_y(L^{p_2,p_2'}(\R^k))} \\
     &\les  2^{j\frac{1+d}{2}-\frac{d-k}{p_1}-\frac{k}{p_2}}    
     \|f\|_{\mathcal L^{p_1}_y(L^{p_2}(\R^k))}   \\
     &= 2^{j\frac{1+d}{2}-\frac{d-k}{p_1}-\frac{k}{p_2}}    
     \|f\|_{\mathcal L^{p_1}_y(\mathcal L^{p_2}_z)}. 
   \end{aligned}
   \end{align}
   In the second line we made use of Corollary~4.5 in \cite{ChenSun2022} once more, and in the fourth line we
   used another nontrivial embedding of interpolation spaces from Theorem~2~(ii) in~\cite{Mandel2023}.
   
   \medskip
   
   Interchanging the order of integration, \eqref{eq:Tj_estimate} allows to prove in a similar way
   \begin{align} \label{eq:InterpolatedBoundsII}
     \|T_j f\|_{(\mathcal L^{p_1}_y)'(\mathcal L^{p_2}_z)'}
     \les 2^{j\frac{1+d}{2}-\frac{d-k}{p_1}-\frac{k}{p_2}}    
     \|f\|_{\mathcal L^{p_1}_y(\mathcal L^{p_2}_z)}. 
   \end{align}
   We thus obtain
   \begin{align} \label{eq:InterpolatedBoundsIII}
     \|T_j f\|_{(\mathcal L^{p_1}_y)'(\mathcal L^{p_2}_z)'}
     +\|T_j f\|_{(\mathcal L^{p_2}_z)'(\mathcal L^{p_1}_y)'}
     \les 2^{j\frac{1+d}{2}-\frac{d-k}{p_1}-\frac{k}{p_2}}    
     \|f\|_{\mathcal L^{p_2}_z(\mathcal L^{p_1}_y)}, 
   \end{align}
   which proves the claim given the formula for the norm of $X_{\vec p}'$.
\end{proof}

It will be convenient to use a simplified version of this result coming from the choice $p_1=p_2$. 
Recall $m=\min\{k,d-k\}$.

\begin{cor}\label{cor:Xpestimates}
  In the case $1\leq p<\frac{2m}{m+1}$ we have
  $$
    \|T_j f\|_{L^{p'}(\R^d)} \les 2^{j(\frac{1+d}{2}-\frac{d}{p})} \|f\|_{L^{p}(\R^d)}. 
  $$
  If $p=\frac{2m}{m+1}$ and $m<\frac{d}{2}$, then we have for all
  $f\in\mathcal S_{G_k}(\R^d)$ 
  $$
    \|T_j f\|_{L^{p',\infty}(\R^d)} 
    \les 2^{j(\frac{1+d}{2}-\frac{d}{p})} \|f\|_{L^{p,1}(\R^d)}. 
  $$
  If $p=\frac{2m}{m+1}$ and $m=\frac{d}{2}$, then we have for all
  $f\in\mathcal S_{G_k}(\R^d)$ 
  $$
    \|T_j f\|_{L^{p',\infty}(\R^m)(L^{p',\infty}(\R^m))} 
    \les 2^{j(\frac{1+d}{2}-\frac{d}{p})} \|f\|_{L^{p,1}(\R^m)(L^{p,1}(\R^m))}. 
  $$  
\end{cor}
\begin{proof}
  We first consider the non-endpoint case $p_1=p_2:=p\in [1,\frac{2m}{m+1})$ and $\vec p:=(p_1,p_2)=(p,p)$.
  Then Lemma~\ref{lem:XpEstimates} applies due to 
  $1\leq p_1,p_2<\frac{2m}{m+1} = \min\{\frac{2k}{k+1},\frac{2(d-k)}{d-k+1}\}$, which proves the claim
  because of $X_{\vec p}=L^p(\R^d)$ and $X_{\vec p}'=L^{p'}(\R^d)$. 
  In the case $p_1=p_2=p=\frac{2m}{m+1}$ and $m<\frac{d}{2}$
  we may w.l.o.g. assume $m=k<d-k$ and obtain the embeddings
  \begin{align} \label{eq:Xvecp} 
    \begin{aligned}
      X_{\vec p} &\supset \mathcal L^p_y(\mathcal L^p_z)= L^p(\R^{d-k})(L^{p,1}(\R^k)), \\
      X_{\vec p}' &\subset (\mathcal L^p_y)'(\mathcal L^p_z)'= L^{p'}(\R^{d-k})(L^{p',\infty}(\R^k)).
  \end{aligned}
  \end{align} 
  Exploiting further embeddings 
\begin{align}\label{eq:inclusions}
  \begin{aligned}
    L^p(\R^{d-k})(L^{p,r}(\R^k)) &\supset L^{p,r}(\R^d) \qquad\text{if }0<r\leq p\leq \infty, \\
    L^q(\R^{d-k})(L^{q,s}(\R^k)) &\subset L^{q,s}(\R^d) \qquad\text{if }0<q\leq s\leq \infty      
  \end{aligned}
\end{align}
 for $r=1,s=\infty$ from \cite[Lemma~3]{Mandel2023}  yields 
 \begin{align*}
      X_{\vec p} \supset  L^{p,1}(\R^d), \qquad
      X_{\vec p}' \subset L^{p',\infty}(\R^d).
  \end{align*}
  So, the estimates from Lemma~\ref{lem:XpEstimates} imply  
  \begin{align*}
     \|T_j f\|_{L^{p',\infty}(\R^d)}
     \les \|T_j f\|_{X_{\vec p}'}  
     \les 2^{j(\frac{1+d}{2}-\frac{d}{p})} \|f\|_{X_{\vec p}}  
     \les 2^{j(\frac{1+d}{2}-\frac{d}{p})} \|f\|_{L^{p,1}(\R^d)}
  \end{align*}
  and the claim is proved.
\end{proof}

Note that in the exceptional case $m=\frac{d}{2}, p=\frac{2m}{m+1}$ embeddings analogous
to~\eqref{eq:inclusions} do in general not hold, see the Lemma
in~\cite{Cwikel1974}. In order to uncover the mapping properties of $T$, which are essentially determined by
those of the linear operator $\sum_{j=1}^\infty T_j$, we need to interpolate the estimates for $T_j$ from
Corollary~\ref{cor:Xpestimates} with the ones from
Lemma~\ref{lem:FourierEstimate}.
In the most difficult case $m=\frac{d}{2}$
we use recently established identities for
real interpolation spaces between mixed Lorentz spaces: Corollary~1 in \cite{Mandel2023} gives for
$1<p_0\neq p_1<\infty$ and $1\leq r,q\leq \infty$ 
\begin{align} \label{eq:md/2Estimate}  
    \Big(L^{p_0,r}(\R^m)(L^{p_0,r}(\R^m)),L^{p_1}(\R^{2m})\Big)_{\theta,q} = L^{p_\theta,q}(\R^{2m})      
\end{align}
whenever $\frac{1}{p_\theta}= \frac{1-\theta}{p_0}+\frac{\theta}{p_1}$ and $0<\theta<1$. 
 
\medskip

\textbf{Proof of Theorem~\ref{thm:RestrictionExtension}:} We have to show that the
Restriction-Extension operator $T:L^p_{G_k}(\R^d)\to L^q_{G_k}(\R^d)$ is bounded if
$$ 
    \min\left\{\frac{1}{p},\frac{1}{q'}\right\} > \frac{d+1}{2d},\qquad 
    \frac{1}{p}-\frac{1}{q} \geq  \frac{2}{d+m}.
$$  
Recall that the boundary of this set if the pentagon $ABDD'B'$ in Figure~\ref{fig:Riesz}.
Since the claim for $m=\min\{k,d-k\}=1$ is already covered by Guti\'{e}rrez' result \cite[Theorem~6]{Gutierrez2004},
  we focus on $m\geq 2$.
  We   first consider the case $m<\frac{d}{2}$. 
  The trivial estimates for $Tf=\mathcal J\ast (\tau(|D|)f)$ exploit the pointwise bounds
  of the kernel function $\mathcal J$ as well as Young's convolution inequality in Lorentz spaces. More
  precisely,
  \begin{equation}\label{eq:TrivialEstimatesT}
    \|Tf\|_{\frac{2d}{d-1},\infty} \leq \|\mathcal J\|_{\frac{2d}{d-1},\infty} \|f\|_1,\quad
    \|Tf\|_{\infty} \leq \|\mathcal J\|_{\frac{2d}{d-1},\infty} \|f\|_{\frac{2d}{d+1},1},\quad
    \|Tf\|_{\infty} \leq \|\mathcal J\|_{\infty} \|f\|_1.
  \end{equation}
  Note that these estimates are located at 
  the  corners A,B,B' in Figure~\ref{fig:Riesz}.
  By real interpolation, it remains to prove restricted weak-type estimates at the   corners
  D,D' situated on the line $\frac{1}{p}-\frac{1}{q}=\frac{2}{d+m}$.
  Thanks to Proposition~\ref{prop:RepresentationFormula} and \eqref{eq:defChi},\eqref{eq:defPhij} it is
  sufficient to prove these estimates for the linear operator $\sum_{j=1}^\infty T_j$. 
  From Lemma~\ref{lem:FourierEstimate} and Corollary~\ref{cor:Xpestimates} we know 
  $$
    \text{(i)}\; \|T_j f\|_2 \les 2^{j\frac{1}{2}}\|f\|_{p_{ST}},\quad
    \text{(ii)}\; \|T_j f\|_{p_{ST}'} \les 2^{j\frac{1}{2}}\|f\|_2,\quad 
    \text{(iii)}\; \|T_j f\|_{\frac{2m}{m-1},\infty} \les 2^{j\frac{m-d}{2m}}\|f\|_{\frac{2m}{m+1},1}.
  $$
  Bourgain's interpolation scheme~\cite[p.604]{CarberySeegerETC1999} with interpolation parameter
  $\theta=\frac{m}{d}\in (0,1)$ applied to (i),(iii) and (ii),(iii), respectively, gives
  $$
    \left\|\sum_{j=1}^\infty T_j f\right\|_{q,\infty} \les \|f\|_{p,1} 
    \quad\text{where }
    \begin{cases}
      \frac{1}{p} = \frac{1-\theta}{p_{ST}}+\frac{\theta}{\frac{2m}{m+1}},\;
      \frac{1}{q} = \frac{1-\theta}{2}+\frac{\theta}{\frac{2m}{m-1}} \quad\text{or }\\
      \frac{1}{p} = \frac{1-\theta}{2}+\frac{\theta}{\frac{2m}{m+1}},\;
      \frac{1}{q} = \frac{1-\theta}{p_{ST}'}+\frac{\theta}{\frac{2m}{m-1}}.
    \end{cases}
  $$
  This is equivalent to
  \begin{equation} \label{eq:endpointestimates}
    \left\|\sum_{j=1}^\infty T_j f\right\|_{q,\infty} \les \|f\|_{p,1} 
    \quad\text{where }
    \begin{cases}
      \frac{1}{q} = \frac{d-1}{2d},\; \frac{1}{p}-\frac{1}{q}=\frac{2}{d+m}  \quad\text{or }\\
      \frac{1}{p} = \frac{d+1}{2d},\; \frac{1}{p}-\frac{1}{q}=\frac{2}{d+m}.
    \end{cases}
  \end{equation}  
  Since these conditions on $p,q$ describe the two corners $D$, $D'$ in the Riesz diagram, this  
  finishes the proof in the case $m<\frac{d}{2}$.
  
  \medskip
  
  In the case $m=\frac{d}{2}$ we still have \eqref{eq:TrivialEstimatesT} and (i),(ii), but (iii) needs to be
  replaced by the estimate 
  $$
    \text{(iii)'}\; \|T_j f\|_{L^{\frac{2m}{m-1},\infty}(\R^m)(L^{\frac{2m}{m-1},\infty}(\R^m))} \les
    2^{j\frac{m-d}{2m}}\|f\|_{L^{\frac{2m}{m+1},1}(\R^m)(L^{\frac{2m}{m+1},1}(\R^m))},
  $$ 
  see Lemma~\ref{lem:XpEstimates}. Once more, the interpolation scheme
  from~\cite[p.604]{CarberySeegerETC1999} shows for $\theta=\frac{m}{d}=\frac{1}{2}$
  \begin{align*}
    \left\|\sum_{j=1}^\infty T_j f\right\|_{
    \big(L^{\frac{2m}{m-1},\infty}(\R^m)(L^{\frac{2m}{m-1},\infty}(\R^m)),L^2(\R^d)\big)_{\theta,\infty}}
    &\les \|f\|_{
    \big(L^{\frac{2m}{m+1},1}(\R^m)(L^{\frac{2m}{m+1},1}(\R^m)),L^{p_{ST}}(\R^d)\big)_{\theta,1}} \\
    \left\|\sum_{j=1}^\infty T_j f\right\|_{
    \big(L^{\frac{2m}{m-1},\infty}(\R^m)(L^{\frac{2m}{m-1},\infty}(\R^m)),L^{p_{ST}'}(\R^d)\big)_{\theta,\infty}}
    &\les \|f\|_{
    \big(L^{\frac{2m}{m+1},1}(\R^m)(L^{\frac{2m}{m+1},1}(\R^m)),L^2(\R^d)\big)_{\theta,1}} 
  \end{align*} 
  The identity \eqref{eq:md/2Estimate} for $r=q=1$ resp. $r=q=\infty$ shows that this is equivalent to 
  \eqref{eq:endpointestimates}. So this finishes the sufficiency proof also in the case $m=\frac{d}{2}$.
  
  \medskip
  
  Our conditions on $p,q$ are in fact optimal. Indeed, the
  constant density on the unit sphere implies the necessity of $\frac{1}{q}<\frac{d-1}{2d}$ and hence
  $\frac{1}{p}>\frac{d+1}{2d}$ by duality.
  Moreover, if   $T:L^p_{G_k}(\R^d)\to L^q_{G_k}(\R^d)$ was bounded for any pair
  $(p,q)$ with $\mu:=\frac{1}{p}-\frac{1}{q}<\frac{2}{d+m}$, 
  then, by symmetry, $T:L^{q'}_{G_k}(\R^d)\to L^{p'}_{G_k}(\R^d)$ would be bounded as well. So the
  Riesz-Thorin Theorem would imply the boundedness of $T:L^{\tilde p}(\R^d)\to L^{\tilde
  p'}(\R^d)$ with $\frac{1}{\tilde p}=\frac{\mu+1}{2}<\frac{d+m+2}{2(d+m)}$. 
  In view of
  $$
    T=S^*S\quad\text{for}\quad
    S:L^{\tilde p}_{G_k}(\R^d)\to L^2(\mathbb{S}^{d-1}),\; f\mapsto \hat f|_{\mathbb{S}^{d-1}}
  $$
  this would in turn imply that $S:L^{\tilde p}_{G_k}(\R^d)\to L^2(\mathbb{S}^{d-1})$ is bounded. However, the
  exponent $\tilde p=\frac{2(d+m)}{d+m+2}$ is largest possible for this inequality by Theorem~1.3(i) in
  \cite{ManDOS}, a contradiction. So the assumption was false and $\frac{1}{p}-\frac{1}{q}\geq \frac{2}{d+m}$
  is proved to be necessary.
\qed

%
 
\section{$G_k$-symmetric Limiting Absorption Principles}

In this section we carry out a related analysis to prove $L^p_{G_k}$-$L^q_{G_k}$-Limiting
Absorption Principles for elliptic (pseudo-)differential operators $P(|D|)$.
Here the task is to determine $p,q\in [1,\infty]$ such that the linear map 
$$
  (P(|D|)+i0)^{-1}u := \lim_{\eps\to 0^+} \mathcal F^{-1}\left( \frac{\hat u}{P(|\cdot|)+i\eps}\right)
$$ 
is well-defined and bounded as an operator from $L^p_{G_k}(\R^d)$ into $L^q_{G_k}(\R^d)$. 
In the non-symmetric setting optimal bounds were found by Kenig-Ruiz-Sogge~\cite{KenigRuizSogge1987} (for
$q=p'$) and Guti\'{e}rrez~\cite{Gutierrez2004} in the special case of the Helmholtz operator
$-\Delta-1$, i.e., $P(r)=r^2-1$.
Our aim is to prove a $G_k$-symmetric counterpart of this result that even applies to a more general class
of symbols. In particular, we significantly improve two earlier contributions~\cite{WethYesil2021,ManDOS}
dealing with the Helmholtz operator assuming $G_k$-symmetry.  
Our assumptions on the symbol are as follows:
\begin{itemize}
  \item[(A)] $P$ is smooth on $[0,\infty)$, $P(0)\neq 0$ and $P$ has finitely many simple zeros
  $r_1,\ldots,r_M$ on $(0,\infty)$. Moreover, there are $R,\eps,s>0$ such that  
  $$
    \left|\frac{d^k}{dr^k}\left(\frac{r^s}{P(r)} \right)\right|
	\les  r^{-k-\eps}
    \qquad\text{for  }  r\geq R \text{ where } k:= \lfloor d/2\rfloor+1.
  $$
\end{itemize}
We emphasize that (A) holds, e.g., for all polynomials $P$ of degree $s\in\N$ with $P(0)\neq 0$ such that all
positive zeros are simple, but also for other physically relevant symbols such as relativistic Schr\"odinger
operators $P(|D|)=(\mu+|D|^2)^{s/2}- \Lambda$ with $\Lambda>\mu>0$. 

\medskip

For symbols $P$ as in (A) we  can choose $\delta>0$ small enough and smooth nonnegative functions
$\chi_0,\ldots,\chi_m$ such that
\begin{align}\label{eq:chim}
  \begin{aligned}
  &\chi_m\in C_0^\infty(\R_+),\qquad \chi_0(r):= 1- \sum_{m=1}^M \chi_m(r)\quad \text{with} \\ 
  &|P'(r)|>0 \text{ on }\supp(\chi_m) \quad\text{where }\chi_m(r)=1 \text{ for
  }|r-r_m|\leq\delta.
\end{aligned}
\end{align}
The assumption regarding the asymptotic behaviour of $P$ at infinity allows to make use of well-known
Bessel potential estimates. In fact, combining (A) with Proposition~1 in~\cite{Mandel2022} we find
that
\begin{equation}\label{eq:multiplier}
  \xi\mapsto \chi_0(|\xi|)(1+|\xi|^2)^{s/2}P(|\xi|)^{-1}
  \quad\text{is an $L^\mu(\R^d)$-multiplier for all }\mu\in [1,\infty].
\end{equation}
In the following we use the principal value operator given by 
$$
  p.v. \int_0^\infty \frac{g(r)}{P(r)} \,dr 
  = \lim_{\delta\to 0^+} \int_{|P(r)|>\delta}  \frac{g(r)}{P(r)} \,dr.
$$
This is motivated by the Plemelj-Sokhotsky formula.

\begin{prop} \label{prop:PlemSokh}
  Assume (A) and $h\in C^1(\R_+)$. Then 
  $$
    \lim_{\eps\to 0^+} \int_0^\infty \frac{\chi_m(r) h(r)}{P(r)+i\eps}\,dr 
    = -\pi i h(r_m) P'(r_m)^{-1} + p.v. \int_0^\infty \frac{\chi_m(r)h(r)}{P(r)}\,dr.
  $$
\end{prop}
\begin{proof}
  For small $\tau>0$ we set $I_\tau := \{r\in \supp(\chi_m):|P(r)|<\tau\}$. 
  By assumption \eqref{eq:chim} we have $|P'|>0$ on $\supp(\chi_m)$ and in particular on $I_\tau$. So we may
  define $$
   \tilde h(s):= \chi_m(P^{-1}(s)) h(P^{-1}(s))(P^{-1})'(s)
  $$ 
  where $P^{-1}:(-\tau,\tau)\to I_\tau$ denotes the local inverse of $P|_{I_\tau}$, so  $P^{-1}(0)=r_m$.
  Then the claim follows from 
  \begin{align*}
      \lim_{\eps\to 0^+} \int_0^\infty \frac{\chi_m(r)h(r)}{P(r)+i\eps}\,dr
      &=  \lim_{\eps\to 0^+} \int_{I_\tau} \frac{\chi_m(r)h(r)}{P(r)+i\eps}\,dr 
      + \int_{I_\tau^c} \frac{\chi_m(r)h(r)}{P(r)}\,dr\\
      &=  
      \lim_{\eps\to 0^+} \int_{-\tau}^{\tau} \frac{\tilde h(s)}{s+i\eps}\,ds
      + \int_{I_\tau^c} \frac{\chi_m(r)h(r)}{P(r)}\,dr \\
      &=  
      \lim_{\eps\to 0^+} \left[\, \tilde h(0) \int_{-\tau}^\tau \frac{1}{s+i\eps}\,ds +
      \int_{-\tau}^\tau \frac{\tilde h(s)-\tilde h(0)}{s+i\eps}\,ds\, \right]
      + \int_{I_\tau^c} \frac{\chi_m(r)h(r)}{P(r)}\,dr\\
      &=  \tilde h(0)  \lim_{\eps\to 0^+}
      \int_{-\tau}^\tau \frac{s-i\eps}{s^2+\eps^2}\,ds 
      +  \int_{-\tau}^\tau  \frac{\tilde h(s)-\tilde h(0)}{s}\,ds 
      + \int_{I_\tau^c} \frac{\chi_m(r)h(r)}{P(r)}\,dr\\
      &= \tilde h(0)  \lim_{\eps\to 0^+}
      \int_{-\tau\eps^{-1}}^{\tau\eps^{-1}} \frac{t-i}{t^2+1}\,dt
      +  \lim_{\delta\to 0^+}  \int_{\delta<|s|<\tau} \frac{\tilde h(s)}{s}\,ds
      + \int_{I_\tau^c} \frac{\chi_m(r)h(r)}{P(r)}\,dr \\
      &=   -\pi i \tilde h(0) +  \lim_{\delta\to 0^+}  \int_{|P(r)|>\delta}  \frac{\chi_m(r)h(r)}{P(r)}\,dr \\
      &= -\pi i h(r_m) P'(r_m)^{-1} + p.v. \int_0^\infty \frac{\chi_m(r)h(r)}{P(r)}\,dr.
  \end{align*}
%
\end{proof}

In the following let $\tau\in C_0^\infty(\R_+)$ denote  a function  such that $\tau$
is identically 1 on $\supp(\chi_m)$ for all $m=1,\ldots,M$, in particular $(1-\chi_0)\tau=1-\chi_0$.

\begin{prop}\label{prop:RepresentationFormulaII}
  Assume (A). Then 
  $$
    (P(|D|)+i0)^{-1}f = Rf + \sum_{m=1}^M \Phi^m \ast (\tau(|D|)f)
  $$ 
  where $R:L^p(\R^d)\to L^q(\R^d)$ is a bounded linear operator whenever $p,q\in [1,\infty]$ satisfy
  $$
    0\leq \frac{1}{p}-\frac{1}{q}\leq \frac{s}{d}
    \quad\text{and}\quad 
    \left(\frac{1}{p},\frac{1}{q}\right)\notin\left\{\left(1,\frac{d-s}{d}\right), 
    \left(\frac{s}{d},0\right)\right\}
  $$
  and where the smooth kernel functions are given by 
  \begin{align*}
     \Phi^m(z) 
    = - i\pi r_m^{d-1}P'(r_m)^{-1} \mathcal J(r_m|z|) + p.v. \int_0^\infty \frac{\chi_m(r) r^{d-1}}{P(r)}
    \mathcal J(r|z|) \,dr.
  \end{align*}
\end{prop}
\begin{proof}
  We define $Rf:= \mathcal F^{-1}(\chi_0(|\cdot|)P(|\cdot|)^{-1}\hat f)$. From~\eqref{eq:multiplier} and
  well-known Bessel potential estimates we get 
  $$
    \|Rf\|_q
    \les \| \mathcal F^{-1}( (1+|\cdot|^2)^{-s/2} \hat f)\|_q
    \les \|f\|_p
  $$
  thanks to our assumptions on $p,q$.
  So we have  $(P(|D|)+i0)^{-1}f = Rf+\Phi\ast (\tau(|D|)f)$ where 
    \begin{align*}
    \Phi(z) 
    &= \lim_{\eps\to 0^+} \mathcal F^{-1}\left(\frac{1-\chi_0(|\cdot|)}{P(|\cdot|)+i\eps}\right)(z)  \\
    &\stackrel{\eqref{eq:chim}}{=} \lim_{\eps\to 0^+}  \sum_{m=1}^M (2\pi)^{-\frac{d}{2}} \int_{\R^d}
    \frac{\chi_m(|\xi|)}{P(|\xi|)+i\eps} e^{iz\cdot \xi}\,d\xi  \\
    &=  \sum_{m=1}^M \lim_{\eps\to 0^+} \int_0^\infty
    \frac{\chi_m(r)r^{d-1}}{P(r)+i\eps} \left((2\pi)^{-\frac{d}{2}} \int_{\mathbb S^{d-1}} e^{iz\cdot
    r\omega}\,d\sigma(\omega) \right) \,dr
    \\
    &\stackrel{\eqref{eq:defJ}}=  \sum_{m=1}^M \lim_{\eps\to 0^+} \int_0^\infty \frac{\chi_m(r)
    r^{d-1}}{P(r)+i\eps} \mathcal J(r|z|)\,dr .
  \intertext{The Plemelj-Sokhotsky formula from Proposition~\ref{prop:PlemSokh} gives}
   \Phi(z)
    &=   \sum_{m=1}^M  
      \left( -i\pi    r_m^{d-1} P'(r_m)^{-1} \mathcal J(r_m|z|)
      + p.v.  \int_0^\infty \frac{\chi_m(r) r^{d-1}}{P(r)} \mathcal J(r|z|)\,dr
    \right),
  \end{align*}
  which proves the claim.
\end{proof}

So we have shown that the mapping properties of $(P(|D|)+i0)^{-1}$ are determined by the mapping properties
of $R$ and the convolution operators with the kernels $\Phi^m$. Our next aim is to show that each $\Phi^m$
has, qualitatively, the same asymptotic expansion as the function $\mathcal J(|\cdot|)$. This provides the
link to the Restriction-Extension operator studied in~Theorem~\ref{thm:RestrictionExtension}.   

\begin{prop}\label{prop:PhimAsymptotics}
  Assume (A) and $L\in\N,m\in\{1,\ldots,M\}$. Then  there are
  $\alpha_{lm},\beta_{lm}\in\C$ for $l\in\{0,\ldots,L-1\}$ such that 
  $$
    \Phi^m(z) = \sum_{l=0}^{L-1} |z|^{\frac{1-d}{2}-l}\left(\alpha_{lm}  
     e^{ir_m |z|} + \beta_{lm} e^{-ir_m|z|}\right) + O(|z|^{\frac{1-d}{2}-L}) 
     \quad\text{as }|z|\to\infty.
  $$
\end{prop}
\begin{proof}
  By Proposition~\ref{prop:RepresentationFormulaII} we have   
  \begin{align*}
     \Phi^m(z) 
    = - i\pi r_m^{d-1}P'(r_m)^{-1} \mathcal J(r_m|z|) + p.v. \int_0^\infty \frac{\chi_m(r) r^{d-1}}{P(r)}
    \mathcal J(r|z|) \,dr.
  \end{align*}
  So  Proposition~\ref{prop:AsymptoticsJ} provides the claimed asymptotic expansion for
  the imaginary part of $\Phi^m(z)$ and it remains to analyze the  principal value integral. To do this
  we use the decomposition
  \begin{equation*} 
    \mathcal J(s) = \mathcal J^1(s)+s^{\frac{1-d}{2}}\mathcal J^2(s)e^{is}+s^{\frac{1-d}{2}}\ov{\mathcal J^2(s)}e^{-is}
  \end{equation*}
  for $\mathcal J^1,\mathcal J^2$ as in~\eqref{eq:decomposition}. 
  Since $0\notin \supp(\chi_m)$ by \eqref{eq:chim} and $\mathcal J^1$ has compact support, we have $\mathcal
  J^1(r|z|)=0$ for large $|z|$ and $r\in\supp(\chi_m)$. So it remains to analyze the integrals
  involving $\mathcal J^2$.  Define  
  $$
    g_m(r,z):= \chi_m(r)(r-r_m) P(r)^{-1} r^{\frac{d-1}{2}}\mathcal J^2(r|z|)
    \quad\text{for }r>0, r\neq r_m. 
  $$  
  Since $r_m$ is a simple zero of $P$ with $|P'|>0$ on $\supp(\chi_m)$, this
  function is smooth on $(0,\infty)\times\R$.
  Choose an even function $\eta\in C_0^\infty(\R)$ with $\eta(z)=1$ for $z$ near 0. Then
  \begin{align*}
    &\; p.v. \int_0^\infty \frac{\chi_m(r) r^{d-1}}{P(r)} (r|z|)^{\frac{1-d}{2}}\mathcal J^2(r|z|)e^{ir|z|}
    \,dr  \\     
    &=  |z|^{\frac{1-d}{2}} \lim_{\delta\to 0} 
    \int_{|P(r)|>\delta} \frac{g_m(r,z)}{r-r_m}  e^{ir|z|}\,dr
    \\
    &=  |z|^{\frac{1-d}{2}}e^{ir_m|z|} \lim_{\delta\to 0} \int_{|P(r)|>\delta}
    \frac{g_m(r,z)}{r-r_m}  e^{i(r-r_m)|z|}\,dr \\
    &=  |z|^{\frac{1-d}{2}}e^{ir_m|z|} \cdot i  g_m(r_m,z) \int_\R \frac{\eta(\rho)}{\rho}\sin(\rho|z|)\,d\rho
       \\ 
     &\; +  |z|^{\frac{1-d}{2}}e^{ir_m|z|}
    \int_\R \frac{g_m(r_m+\rho,z)-g_m(r_m,z)\eta(\rho)}{\rho}  e^{i\rho|z|}\,d\rho. 
  \end{align*}
  We first investigate the asymptotic expansion of the first term.
  The function $A(s):= \int_\R \rho^{-1}\eta(\rho)\sin(\rho s) \,d\rho$ has the property that its
  derivative $A'(s) =  \int_\R \cos(\rho s)\eta(\rho)\,d\rho$ is a Schwartz function. Moreover,
  \begin{align*}
    \lim_{z\to\infty} A(z)
    &= \lim_{z\to\infty} \int_0^z  A'(s)\,ds \\
    &= \lim_{z\to\infty}   \int_0^z   \int_\R  \cos(\rho s)\eta(\rho)\,d\rho  \,ds   \\
    &= \lim_{z\to\infty}   \int_0^z  \left( - \int_\R s^{-1}\sin(\rho s)\eta'(\rho)\,d\rho\right)
    \,ds   \\
    &= -2 \lim_{z\to\infty}    \int_0^\infty  \left(\int_0^z s^{-1} \sin(\rho s)\,ds\right) \eta'(\rho)\,d\rho  
    \\
    &= - 2 \lim_{z\to\infty}    \int_0^\infty  \left(\int_0^{\rho z} \tau^{-1}\sin(\tau)\,d\tau\right)
    \eta'(\rho)\,d\rho  \\
    &= -  2 \int_0^\infty  \eta'(\rho)\,d\rho  \cdot \frac{\pi}{2} \\
    &=  \pi.  
  \end{align*}  
  Hence, for any given $L\in\N$,
  $$
    \int_\R \frac{\eta(\rho)}{\rho}\sin(\rho|z|)\,d\rho
    = A(|z|)
    = \pi  -  \int_{|z|}^\infty A'(t)\,dt
    = \pi + O(|z|^{-L}).  
  $$
  This shows
  \begin{align*}
     |z|^{\frac{1-d}{2}}e^{ir_m|z|} \cdot i  g_m(r_m,z) \int_\R \frac{\eta(\rho)}{\rho}\sin(\rho|z|)\,d\rho
    &=  i|z|^{\frac{1-d}{2}}e^{ir_m|z|}  g_m(r_m,z)\cdot \left(\pi  + O(|z|^{-L})\right). 
  \end{align*}
  So it remains to examine the asymptotics of $g_m(r_m,z)=P'(r_m)^{-1}r_m^{\frac{d-1}{2}}\mathcal
  J^2(r_m|z|)$ with respect to $z$. So  
  $\mathcal J^2(s)= \sum_{l=0}^{L-1} \alpha_l s^{-l} + O(|s|^{-L})$ as $s\to\infty$ proves the desired
  asymptotic expansion for the term 
  $$
    |z|^{\frac{1-d}{2}}e^{ir_m|z|} \cdot i  g_m(r_m,z) \int_\R \frac{\eta(\rho)}{\rho}\sin(\rho|z|)\,d\rho.
  $$
   
  \medskip
  
  To prove this for the term
  $$
    |z|^{\frac{1-d}{2}}e^{ir_m|z|}
    \int_\R \frac{g_m(r_m+\rho,z)-g_m(r_m,z)\eta(\rho)}{\rho}  e^{i\rho|z|}\,d\rho
  $$
  note that the function 
  $$
    \tilde g_m(\rho,z) := \frac{g_m(r_m+\rho,z)-g_m(r_m,z)\eta(\rho)}{\rho}
  $$ 
  is smooth with compact support. Moreover, the smoothness of $P$, $|(\mathcal J^2)^{(l)}(s)|\les s^{-l}$
  for all $s\in\R$ and $l\in\N_0$ implies that all derivatives of $\tilde g_m(\rho,z)$ with respect to $\rho$ 
  are uniformly bounded with respect to $\rho$ and $z$. Hence, integration by parts gives 
  \begin{align*}
    \left|\int_\R  \tilde g_m(\rho,z) e^{i\rho|z|} \,d\rho\right|
    &= |z|^{-L} \left|\int_\R  \tilde g_m(\rho,z) \frac{d^L}{d\rho^L} (e^{i\rho|z|}) \,d\rho\right| \\
    &= |z|^{-L} \left|\int_\R  \frac{d^L}{d\rho^L}  \big(\tilde g_m(\rho,z)\big)  e^{i\rho|z|}  \,d\rho\right|
    \\
    &\les |z|^{-L}.
  \end{align*}   
  So the second term has the claimed asymptotic expansion as well, which finishes the proof.
\end{proof}

\medskip

\begin{thm}\label{thm:resolvent2}
  Assume $d\in\N, k\in\{1,\ldots,d-1\}$ and (A). Then $(P(|D|)+i0)^{-1}:L^p_{G_k}(\R^d)\to
  L^q_{G_k}(\R^d)$ is a bounded linear operator provided that $p,q\in [1,\infty]$ satisfy
  $$
    \min\left\{\frac{1}{p},\frac{1}{q'}\right\}>\frac{d+1}{2d},\quad 
    \frac{2}{d+m}\leq  \frac{1}{p}-\frac{1}{q}\leq \frac{s}{d},\quad 
    \left(\frac{1}{p},\frac{1}{q}\right)\notin 
    \left\{ \left(1,\frac{d-s}{d}\right),\left(\frac{s}{d},0\right)\right\}.
  $$
\end{thm} 
\begin{proof} 
  Proposition~\ref{prop:RepresentationFormulaII} shows 
 $$
   (P(|D|)+i0)^{-1}f=Rf + \sum_{m=1}^M \Phi^m\ast (\tau(|D|)f)
 $$
  where $R:L^p_{G_k}(\R^d)\to  L^q_{G_k}(\R^d)$ is bounded. So it remains 
to show that $f\mapsto \Phi^m\ast (\tau(|D|)f)$
  is bounded from $L^p_{G_k}(\R^d)$ to $L^q_{G_k}(\R^d)$ for exponents $p,q\in [1,\infty]$ such that
\begin{equation} \label{eq:pqconditions}
   \min\left\{\frac{1}{p},\frac{1}{q'}\right\}>\frac{d+1}{2d},\qquad 
    \frac{1}{p}-\frac{1}{q}\geq \frac{2}{d+m}.
  \end{equation}   
  By Proposition~\ref{prop:PhimAsymptotics} the kernel functions $\Phi^m$ satisfy 
  the asymptotic expansions 
  $$
    \Phi^m(z) = \sum_{l=0}^{L-1} |z|^{\frac{1-d}{2}-l}\left(\alpha_{lm}  
     e^{ir_m |z|} + \beta_{lm} e^{-ir_m|z|}\right) + O(|z|^{\frac{1-d}{2}-L}) 
     \quad\text{as }|z|\to\infty.
  $$
  This is qualitatively the same asymptotic expansion as the one of the kernel function $\mathcal J(|\cdot|)$
  of the Restriction-Extension operator, see 
  Proposition~\ref{prop:AsymptoticsJ} and~\eqref{eq:BesselJAsymptotics}. In order to carry out    
  an analogous analysis we define $\mathcal T f:= \Phi^m  \ast (\tau(|D|)f)$. As before, 
  we decompose this operator dyadically using the partition of unity introduced in \eqref{eq:defChi}.
  This leads to
  $\mathcal T=\sum_{j=0}^\infty \mathcal T_j$ where $\mathcal T_0:L^p_{G_k}(\R^d)\to L^q_{G_k}(\R^d)$ is
  bounded whenever $1\leq p\leq q\leq \infty$ and, for $j\geq 1$, 
  $$    
    \mathcal T_j f:= \Phi_j^m  \ast (\tau(|D|)f)
     \quad\text{with}\quad \Phi_j^m(z) :=  \chi(2^{-j}|z|) \Phi^m(z). 
  $$  
  Given the asymptotic expansion of $\Phi^m$ from above we find as in~\eqref{eq:TrivialEstimatesT}
  \begin{equation*}
    \|\mathcal Tf\|_{\frac{2d}{d-1},\infty} \les   \|f\|_1,\qquad
    \|\mathcal Tf\|_{\infty} \les \|f\|_{\frac{2d}{d+1},1},\qquad
    \|\mathcal Tf\|_{\infty} \les \|f\|_1.
  \end{equation*}
  The proof of Lemma~\ref{lem:FourierEstimate} also implies  
  $$
    \text{(i)}\quad \|\mathcal T_j f\|_2 \les 2^{j\frac{1}{2}}\|f\|_{p_{ST}},\qquad
    \text{(ii)}\quad \|\mathcal T_j f\|_{p_{ST}'} \les 2^{j\frac{1}{2}}\|f\|_2 
  $$  
  Finally, pointwise bounds for the corresponding kernel
  function analogous to Proposition~\ref{prop:KjBound} lead to 
  \begin{align*}
    &\text{(iii)}\quad \|\mathcal T_j f\|_{\frac{2m}{m-1},\infty} \les 
    2^{j\frac{m-d}{2m}}\|f\|_{\frac{2m}{m+1},1}
    &&\text{if } m<\frac{d}{2}, \\ 
    &\text{(iii)'}\quad \|\mathcal T_j f\|_{L^{\frac{2m}{m-1},\infty}(\R^m)(L^{\frac{2m}{m-1},\infty}(\R^m))}
    \les 2^{j\frac{m-d}{2m}}\|f\|_{L^{\frac{2m}{m+1},1}(\R^m)(L^{\frac{2m}{m+1},1}(\R^m))}
    &&\text{if } m=\frac{d}{2}.
  \end{align*}
  The same interpolation as in the proof of Theorem~\ref{thm:RestrictionExtension} shows    
  that $\mathcal T:L^p_{G_k}(\R^d)\to L^q_{G_k}(\R^d)$ is bounded whenever
  \eqref{eq:pqconditions} holds. This finishes the proof.
  \end{proof}


\section{Appendix: An oscillatory integral estimate}

In this section we prove estimates for oscillatory integrals of the form
$$
  I_\lambda := \int_{-1}^1 \int_{-1}^1 (1-s_1^2)^{\alpha_1} (1-s_2^2)^{\alpha_2} m(s)
  \chi(\Psi(s)) e^{i\lambda \Psi(s)} \,ds_1\,ds_2
$$ 
that we used in the proof of Proposition~\ref{prop:KjBound}, see~\eqref{eq:oscint_estimate}.
Here and in the following we always assume $\alpha_1,\alpha_2>-1,\lambda\in\R, m\in C^\infty([-1,1]^2)$ and
$\chi\in C_0^\infty(\R)$ with $\supp(\chi)\subset [\frac{1}{2},2]$.
The phase function will be given by  
\begin{equation} \label{eq:Psi}
  \Psi(s):=\sqrt{A-B_1s_1-B_2s_2}
   \qquad\text{where }
   A,B_1,B_2\in\R,\; |B_1|+|B_2|\leq A.  
\end{equation}
We use the shorthand notation $I_\lambda \in\mathcal J_{\lambda,\alpha_1,\alpha_2}$ to say that
$I_\lambda$ is of this form. Our goal is to prove the estimate
\begin{equation*}
  |I_\lambda|  
   \les \rho_1^{\alpha_1+1} \rho_2^{\alpha_2+1}
   \quad\text{where }
   \rho_1 := \min\{1,|\lambda B_1|^{-1}\},\; \rho_2 := \min\{1,|\lambda B_2|^{-1}\},
\end{equation*}
which is uniform with respect to $|\lambda|\geq 1$ and $A,B_1,B_2$ as in \eqref{eq:Psi}, see
Theorem~\ref{thm:OscillatoryIntegral} below. In the following the parameters $\alpha_1,\alpha_2>-1$ will be
considered as fixed, i.e., the constants involved in $\les$ depend on $\alpha_1,\alpha_2\in (-1,\infty)$ in a
continuous way.

\medskip

We shall exploit   that $\Psi$ is smooth on $\Omega:=\{s\in [-1,1]^2:
\frac{1}{2}\leq \Psi(s)\leq 2\}$ with
 \begin{equation} \label{eq:properties_Psi}
  \partial_j \Psi(s) = -\frac{B_j}{2\Psi(s)} \quad \text{ for }s\in\Omega \text{ and }j=1,2.
\end{equation} 
This will be crucial  to set up an integration by parts scheme that is based on
\begin{equation} \label{eq:PsiIntByParts}
  \eta(\Psi(s))  e^{i\lambda\Psi(s)} 
  = \frac{1}{i\lambda B_j}\cdot \tilde\eta(\Psi(s)) \frac{\partial}{\partial s_j}\left(e^{i\lambda\Psi(s)}
  \right)  \qquad\text{for }j=1,2
\end{equation}
where $\tilde\eta(z):= -2z\eta(z)$. In particular, if $\eta$ is smooth with support in $[\frac{1}{2},2]$ then
so is $\tilde \eta$ and the procedure may be iterated.

\medskip

 The analysis of the integrals $I_\lambda$ is lengthy.  
 We start with estimates that do not take the
 oscillatory nature into account. Here we use, for all $\lambda\in\R$ and $s_1,s_2\in [-1,1]$,
 \begin{equation}\label{eq:nonosc_estimate}
   \big|m(s) \chi(\Psi(s)) e^{i\lambda \Psi(s)}\big|
   \leq \|m\|_\infty \|\chi\|_\infty \ind_{\Psi(s)\leq 2}
   \les \ind_{|s_1|\geq \frac{A-|B_2|-4}{|B_1|}}\ind_{|s_2|\geq \frac{A-|B_1|-4}{|B_2|}}.
 \end{equation}
 Indeed,  $\Psi(s)\leq 2$  implies  $|B_1||s_1|\geq  A-|B_2|-4$ and $|B_2||s_2|\geq  A-|B_1|-4$.

\begin{prop}\label{prop:OscIntEstimatesI}
  Let $\alpha_1,\alpha_2>-1$ and $A,B_1,B_2$ as in \eqref{eq:Psi}. Then  
  \begin{align*}
    \sup_{s_1\in [-1,1]}   \int_{-1}^1
    (1-s_2^2)^{\alpha_2} \ind_{\Psi(s)\leq 2} \,ds_2 
    &\les \min\{1,|B_2|^{-\alpha_2-1}\}, \\
    \int_{-1}^1 \int_{-1}^1 (1-s_1^2)^{\alpha_1} (1-s_2^2)^{\alpha_2} 
    \ind_{\Psi(s)\leq 2}  \,ds_1\,ds_2  
    &\les \min\{1,|B_1|^{-\alpha_1-1}\}\min\{1,|B_2|^{-\alpha_2-1}\}.
  \end{align*}
\end{prop} 
\begin{proof}
  We obtain from \eqref{eq:nonosc_estimate}  
  \begin{align*}
    \begin{aligned}
     \int_{-1}^1  (1-s_2^2)^{\alpha_2} \ind_{\Psi(s)\leq 2} \,ds_2  
    &\stackrel{\eqref{eq:nonosc_estimate}}\les 
      \ind_{|s_1|\geq \frac{A-|B_2|-4}{|B_1|}} \cdot  \int_{-1}^1  
       (1-s_2^2)^{\alpha_2}  \ind_{|s_2|\geq \frac{A-|B_1|-4}{|B_2|}} \,ds_2 \\
    &\les 
      \ind_{|s_1|\geq \frac{A-|B_2|-4}{|B_1|}} \cdot  \int_{\frac{(A-|B_1|-4)_+}{|B_2|}}^1  
             (1-s_2)^{\alpha_2} \,ds_2    \\
    &\les   \ind_{|s_1|\geq \frac{A-|B_2|-4}{|B_1|}} \cdot  \left(1 -
    \frac{(A-|B_1|-4)_+}{|B_2|}\right)^{\alpha_2+1} \\ 
    &\les   \ind_{|s_1|\geq \frac{A-|B_2|-4}{|B_1|}} \cdot  \min\{1,|B_2|^{-\alpha_2-1}\}. 
  \end{aligned}
  \end{align*}
  In the last line we used  $|B_2|-(A-|B_1|-4)_+ \leq |B_2|-(A-|B_1|-4)\leq 4$. So 
  the first inequality is immediate and the second one results from multiplying this inequality with
  $(1-s_1^2)^{\alpha_1}$ and integration over $(-1,1)$.
\end{proof}

We now use this estimates to prove preliminary versions of our estimates for $I_\lambda$. We first focus on
oscillations with respect to~$s_2$. Here and in the following we shall often replace the symbols $m,\chi$ by
indexed versions or tilde versions to indicate functions with the same qualitative properties, namely those 
mentioned right before~\eqref{eq:Psi}.

\begin{prop}\label{prop:OscIntEstimatesIb}
  Let $\alpha_1,\alpha_2>-1$ and $A,B_1,B_2$ as in \eqref{eq:Psi}. Then, for  $|\lambda|\geq 1$,
  \begin{align*}
     \sup_{s_1\in [-1,1]} \left| \int_{-1}^1
    (1-s_2^2)^{\alpha_2}m(s)\chi(\Psi(s))e^{i\lambda\Psi(s)}\,ds_2\right|
    &\les \rho_2^{\alpha_2+1}, \\
    \left|\int_{-1}^1 \int_{-1}^1 (1-s_1^2)^{\alpha_1} (1-s_2^2)^{\alpha_2} 
    m(s)\chi(\Psi(s))e^{i\lambda\Psi(s)} \,ds_1\,ds_2\right|  
    &\les\min\{1,|B_1|^{-\alpha_1-1}\} \cdot \rho_2^{\alpha_2+1}.
  \end{align*}  
\end{prop}
\begin{proof}
  In view of Proposition~\ref{prop:OscIntEstimatesI} there is nothing left to prove for $\rho_2=1$, so we 
  assume $0<\rho_2=|\lambda B_2|^{-1}<1$. As before, we fix $s_1\in [-1,1]$ and investigate the
  one-dimensional integrals with respect to $s_2$. It will turn out convenient to subdivide the domain of integration according to $|1-s_2^2|\leq
  \rho_2$ or $|1-s_2^2|> \rho_2$. For the integral over the former region we use a simple pointwise estimate
  and Proposition~\ref{prop:OscIntEstimatesI}:
  \begin{align*}
    \left| \int_{|1-s_2^2|\leq \rho_2} 
    (1-s_2^2)^{\alpha_2}m(s)\chi(\Psi(s))e^{i\lambda\Psi(s)}\,ds_2\right|   
    &\stackrel{\eqref{eq:nonosc_estimate}}\les 
     \ind_{|s_1|\geq \frac{A-|B_2|-4}{|B_1|}}  \int_{1-s_2^2\leq \rho_2} (1-s_2^2)^{\alpha_2}
     \ind_{\Psi(s)\leq 2}\,ds_2   \\
    &\les \ind_{|s_1|\geq \frac{A-|B_2|-4}{|B_1|}} \;\rho_2^{\alpha_2+1}. 
  \end{align*}  
  As in the previous proposition, this implies both estimates for this part of the integral.  
  For the other part we use integration by parts.  
  \begin{align} \label{eq:OscIntIa}
      \begin{aligned}
    &\; \left| \int_{|1-s_2^2|> \rho_2} 
    (1-s_2^2)^{\alpha_2}m(s)\chi(\Psi(s))e^{i\lambda\Psi(s)}\,ds_2\right|   \\
    &\stackrel{\eqref{eq:PsiIntByParts}}= |\lambda B_2|^{-1} \left|\int_{|1-s_2^2|> \rho_2} 
    (1-s_2^2)^{\alpha_2}m(s)\chi_1(\Psi(s)) \frac{\partial}{\partial s_2}
    \left(e^{i\lambda\Psi(s)}\right)\,ds_2\right|
    \\
    &\les |\lambda B_2|^{-1}
     \left(
    \rho_2^{\alpha_2}     + \left|\int_{|1-s_2^2|> \rho_2}
    \frac{\partial}{\partial s_2}\Big( (1-s_2^2)^{\alpha_2}m(s) \chi_1(\Psi(s))\Big)
    e^{i\lambda\Psi(s)}\,ds_2\right|\right)
    \\
    &\stackrel{\eqref{eq:properties_Psi}}\les |\lambda B_2|^{-\alpha_2-1}   +
    |\lambda B_2|^{-1} |\alpha_2|
      \left|\int_{|1-s_2^2|> \rho_2}
     (1-s_2^2)^{\alpha_2-1}  m_1(s) \chi_1(\Psi(s))  e^{i\lambda\Psi(s)}\,ds_2\right| \\
     &+ |\lambda B_2|^{-1} \left|\int_{|1-s_2^2|> \rho_2}
     (1-s_2^2)^{\alpha_2}  m_2(s) \chi_1(\Psi(s)) 
     e^{i\lambda\Psi(s)}\,ds_2\right| \\     
     &+  |\lambda|^{-1} \left|\int_{|1-s_2^2|> \rho_2}
     (1-s_2^2)^{\alpha_2} m(s)  \chi_2(\Psi(s))  e^{i\lambda\Psi(s)}\,ds_2\right| 
     \end{aligned}
   \end{align}
    where $\chi_1(z):=\chi(z)z/2$, $m_1(s):=-2s_2m(s)$, $m_2(s):=\frac{\partial m}{\partial s_2}$, 
    $\chi_2(z):= \chi_1'(z)/(2z)$.
    
    \medskip
    
    \textbf{1st case $\alpha_2\leq 0$:}\; From~\eqref{eq:OscIntIa},
    Proposition~\ref{prop:OscIntEstimatesI} and $|\lambda|\geq 1$ we get 
    \begin{align*} 
    &\; \left| \int_{|1-s_2^2|> \rho_2} 
    (1-s_2^2)^{\alpha_2}m(s)\chi(\Psi(s))e^{i\lambda\Psi(s)}\,ds_2\right|   \\
    &\les |\lambda B_2|^{-\alpha_2-1}    + |\lambda B_2|^{-1}  |\alpha_2|  \int_{|1-s_2^2|> \rho_2}
     (1-s_2^2)^{\alpha_2-1} \,ds_2   + |\lambda B_2|^{-1}  +  |\lambda|^{-1} |B_2|^{-\alpha_2-1}      \\
     &\les |\lambda B_2|^{-\alpha_2-1}  + |\lambda B_2|^{-1} \rho_2^{\alpha_2}  \\
     &\les  |\lambda B_2|^{-\alpha_2-1}.  
  \end{align*}
  Note that the presence of $|\alpha_2|$ ensures that the singular integral is uniformly bounded from above
  for $\alpha_2\in (-1,0]$.
  Since the integral vanishes identically for $|s_1|<\frac{A-|B_2|-4}{|B_1|}$, see \eqref{eq:nonosc_estimate},
  we even obtain
  \begin{align} \label{eq:OscIntIb}
     \left| \int_{|1-s_2^2|> \rho_2} 
    (1-s_2^2)^{\alpha_2}m(s)\chi(\Psi(s))e^{i\lambda\Psi(s)}\,ds_2\right|     
     &\les  \ind_{|s_1|\geq \frac{A-|B_2|-4}{|B_1|}} \;\rho_2^{\alpha_2+1}.  
  \end{align}
   As in the previous Proposition, this estimates implies both   inequalities and finishes the proof
    for $\alpha_2\leq 0$. 
    
    \medskip
    
    \textbf{2nd case $\alpha_2>0$:} Using integration by parts as in \eqref{eq:OscIntIa}   we get
  \begin{align*}
    &\; \left| \int_{|1-s_2^2|> \rho_2} 
    (1-s_2^2)^{\alpha_2}m(s)\chi(\Psi(s))e^{i\lambda\Psi(s)}\,ds_2\right|   \\
    &\les |\lambda B_2|^{-\alpha_2-1}  +  |\lambda B_2|^{-1} \left|\int_{|1-s_2^2|> \rho_2}
     (1-s_2^2)^{\alpha_2}  m_3(s) \chi_1(\Psi(s)) 
     e^{i\lambda\Psi(s)}\,ds_2\right|
      \\
      & +    
    |\lambda|^{-1} \left|\int_{|1-s_2^2|> \rho_2}
     (1-s_2^2)^{\alpha_2} m(s) \chi_2(\Psi(s))  e^{i\lambda\Psi(s)}\,ds_2\right|
  \end{align*}
  where $m_3(s):= -2\alpha_2s_2m(s)+(1-s^2)\tfrac{\partial m}{\partial s_2}$.
  By iteration we find after finitely many steps, still using
  Proposition~\ref{prop:OscIntEstimatesI},
  \begin{align*}
    &\; \left| \int_{|1-s_2^2|> \rho_2} 
    (1-s_2^2)^{\alpha_2}m(s)\chi(\Psi(s))e^{i\lambda\Psi(s)}\,ds_2\right|   \\
    &\les |\lambda B_2|^{-\alpha_2-1}   +  
    \sum_{\tau=0}^{\lceil\alpha_2\rceil+1}
    |\lambda B_2|^{-\tau} |\lambda|^{-\lceil\alpha_2\rceil-1+\tau}
     \left|\int_{|1-s_2^2|> \rho_2}
     (1-s_2^2)^{\alpha_2-\tau} \tilde m_\tau(s) \tilde\chi_\tau(\Psi(s))  e^{i\lambda\Psi(s)}\,ds_2\right| \\
     &\stackrel{\eqref{eq:nonosc_estimate}}\les  |\lambda B_2|^{-\alpha_2-1}  + 
     \sum_{\tau=0}^{\lceil\alpha_2\rceil+1} |\lambda B_2|^{-\tau} |\lambda|^{-\lceil\alpha_2\rceil-1+\tau}  \int_{|1-s_2^2|> \rho_2}
     (1-s_2^2)^{\alpha_2-\tau} \ind_{\Psi(s)\leq 2}  \,ds_2  \\
    &\les |\lambda B_2|^{-\alpha_2-1}   +  \sum_{\tau=0}^{\lceil\alpha_2\rceil}
    |\lambda B_2|^{-\tau} |\lambda|^{- \lceil\alpha_2\rceil-1 +\tau}  \cdot |B_2|^{-\alpha_2+\tau-1} 
    + |\lambda B_2|^{-\lceil\alpha_2\rceil-1} \cdot  \rho_2^{\alpha_2-\tau+1} \\ 
    &\les |\lambda B_2|^{-\alpha_2-1}   + \sum_{\tau=0}^{\lceil\alpha_2\rceil+1}
    |\lambda B_2|^{-\tau}   |\lambda B_2|^{-\alpha_2-1+\tau}  \\
    &\les |\lambda B_2|^{-\alpha_2-1} .    
  \end{align*}
  As above, this actually  implies the better bound 
  $$
     \left| \int_{|1-s_2^2|> \rho_2} 
    (1-s_2^2)^{\alpha_2}m(s)\chi(\Psi(s))e^{i\lambda\Psi(s)}\,ds_2\right|   \\
    \les  \ind_{|s_1|\geq \frac{A-|B_2|-4}{|B_1|}} \;\rho_2^{\alpha_2+1}, 
  $$ 
  which proves both inequalities. 
\end{proof}

In our next result we prove the claim assuming that both exponents $\alpha_1,\alpha_2$ are nonpositive.

\begin{prop} \label{prop:OscIntEstimatesIII}
  Assume $I_\lambda\in \mathcal J_{\lambda,\alpha_1,\alpha_2}$ for 
  $\alpha_1,\alpha_2>-1$ and $A,B_1,B_2$ as in \eqref{eq:Psi}. 
   Additionally assume $\min\{\alpha_1,\alpha_2\}\leq 0$. Then, for $|\lambda|\geq 1$,  
  $$
    |I_\lambda|\les \rho_1^{\alpha_1+1}\rho_2^{\alpha_2+1}. 
  $$
\end{prop}
\begin{proof} 
  Assume w.l.o.g. $\alpha_1\leq 0$. In view of Proposition~\ref{prop:OscIntEstimatesIb} it suffices to prove
  the claim for $\rho_1=|\lambda B_1|^{-1}<1$. 
  We   split the
  integral $I_\lambda$ according to $I_\lambda=I_\lambda^1+I_\lambda^2$ where 
  \begin{align*}
    I_\lambda^1
    &:= \int_{-1}^1 \int_{1-s_1^2\leq \rho_1} (1-s_1^2)^{\alpha_1} 
  (1-s_2^2)^{\alpha_2}m(s)\chi(\Psi(s))e^{i\lambda\Psi(s)}\,ds_1\,ds_2, \\
  I_\lambda^2
    &:= \int_{-1}^1 \int_{1-s_1^2> \rho_1} (1-s_1^2)^{\alpha_1} 
   (1-s_2^2)^{\alpha_2}m(s)\chi(\Psi(s))e^{i\lambda\Psi(s)}\,ds_1\,ds_2. 
  \end{align*}  
  Proposition~\ref{prop:OscIntEstimatesIb} gives  
\begin{align*}
  |I_\lambda^1|
  &\leq \int_{1-s_1^2\leq \rho_1} (1-s_1^2)^{\alpha_1}\,ds_1 \cdot 
  \sup_{s_1\in [-1,1]} \left| \int_{-1}^1
  (1-s_2^2)^{\alpha_2}m(s)\chi(\Psi(s))e^{i\lambda\Psi(s)}\,ds_2\right| \\
  &\les \int_{1-s_1^2\leq \rho_1} (1-s_1^2)^{\alpha_1}\,ds_1 \cdot 
  \rho_2^{\alpha_2+1} \\ 
  &\les \rho_1^{\alpha_1+1} \rho_2^{\alpha_2+1}. 
\end{align*}  
 So it  remains to estimate $I_\lambda^2$. We use integration by parts  to get
 \begin{align*}
  |I_\lambda^2|
  &\stackrel{\eqref{eq:PsiIntByParts}}= 
  |\lambda B_1|^{-1} \left| \int_{-1}^1  \int_{1-s_1^2> \rho_1}
   (1-s_1^2)^{\alpha_1} (1-s_2^2)^{\alpha_2}m(s) \chi_1(\Psi(s)) \frac{\partial}{\partial
  s_1}\left(e^{i\lambda\Psi(s)}\right)\,ds_1 \,ds_2  \right| \\
  &\les
  |\lambda B_1|^{-1} \cdot
  \sup_{1-s_1^2=\rho_1}
  \left|(1-s_1^2)^{\alpha_1} \int_{-1}^1 (1-s_2^2)^{\alpha_2}m(s)  \chi_1(\Psi(s))
  e^{i\lambda\Psi(s)}\,ds_2\right|  \\
   &+ |\lambda B_1|^{-1} \left| \int_{-1}^1  \int_{1-s_1^2> \rho_1}
   \frac{\partial}{\partial
  s_1}\Big( (1-s_1^2)^{\alpha_1} (1-s_2^2)^{\alpha_2}m(s)\chi_1(\Psi(s))\Big)
  e^{i\lambda\Psi(s)}\,ds_1 \,ds_2  \right| \\  
  &\stackrel{\eqref{eq:properties_Psi}}\les |\lambda B_1|^{-1}\cdot \rho_1^{\alpha_1} \rho_2^{\alpha_2+1} \\
  &+  |\lambda B_1|^{-1}  |\alpha_1| \int_{1-s_1^2> \rho_1} 
    (1-s_1^2)^{\alpha_1-1} \,ds_1 \cdot \sup_{s_1\in [-1,1]} \left|\int_{-1}^1 (1-s_2^2)^{\alpha_2} 
  m_1(s)  \chi_1(\Psi(s)) e^{i\lambda\Psi(s)} \,ds_2\right| \\
  &+  |\lambda B_1|^{-1}   \int_{1-s_1^2> \rho_1} 
    (1-s_1^2)^{\alpha_1} \,ds_1 \cdot \sup_{s_1\in [-1,1]} \left|\int_{-1}^1 (1-s_2^2)^{\alpha_2} 
  m_2(s)  \chi_1(\Psi(s)) e^{i\lambda\Psi(s)} \,ds_2\right| \\
  &+  |\lambda|^{-1}   \left|
  \int_{-1}^1 \int_{1-s_1^2>\rho_1}  (1-s_1^2)^{\alpha_1} (1-s_2^2)^{\alpha_2}  m(s)  \chi_2(\Psi(s))
  e^{i\lambda\Psi(s)}\,ds_2\,ds_1\right|.
\end{align*} 
Using the assumption  $\alpha_1\leq 0$ we get from Proposition~\ref{prop:OscIntEstimatesIb}
\begin{align*}
  |I_\lambda^2|
  &\les  \rho_1^{\alpha_1+1} \rho_2^{\alpha_2+1} +  
   |\lambda B_1|^{-1}  \rho_1^{\alpha_1}  \cdot \rho_2^{\alpha_2+1} 
   + |\lambda B_1|^{-1}   \cdot \rho_2^{\alpha_2+1}
   \\
  &+  |\lambda|^{-1}   \left|
  \int_{-1}^1 \int_{1-s_1^2>\rho_1}  (1-s_1^2)^{\alpha_1} (1-s_2^2)^{\alpha_2} m(s)  \chi_2(\Psi(s))
  e^{i\lambda\Psi(s)}\,ds_2\,ds_1\right| \\
  &\les  \rho_1^{\alpha_1+1} \rho_2^{\alpha_2+1}  +  
  |\lambda|^{-1}   \left|
  \int_{-1}^1 \int_{1-s_1^2>\rho_1}  (1-s_1^2)^{\alpha_1} (1-s_2^2)^{\alpha_2} m(s)  
  \chi_2(\Psi(s)) e^{i\lambda\Psi(s)}\,ds_2\,ds_1\right|.
\end{align*}
Iterating this finitely many times gives, with the aid of Proposition~\ref{prop:OscIntEstimatesI},
\begin{align*}
  |I_\lambda^2|
  &\les  \rho_1^{\alpha_1+1} \rho_2^{\alpha_2+1} +  
  |\lambda|^{-2-\lceil\alpha_1\rceil-\lceil\alpha_2\rceil}   \left|
  \int_{-1}^1 \int_{1-s_1^2>\rho_1}  (1-s_1^2)^{\alpha_1} (1-s_2^2)^{\alpha_2} m(s)
   \chi_3(\Psi(s)) e^{i\lambda\Psi(s)} 
  \,ds_2\,ds_1\right| \\
  &\les  \rho_1^{\alpha_1+1} \rho_2^{\alpha_2+1} +  
  |\lambda|^{-2-\alpha_1-\alpha_2}  
  \int_{-1}^1 \int_{-1}^1  (1-s_1^2)^{\alpha_1} (1-s_2^2)^{\alpha_2} \ind_{\Psi(s)\leq 2}  
  \,ds_2\,ds_1 \\
  &\les  \rho_1^{\alpha_1+1} \rho_2^{\alpha_2+1} +  
  |\lambda|^{-2-\alpha_1 - \alpha_2 } \cdot |B_1|^{-\alpha_1-1} |B_2|^{-\alpha_2-1} \\
  &\les \rho_1^{\alpha_1+1}\rho_2^{\alpha_2+1}. 
\end{align*}
\end{proof}
  
Proposition~\ref{prop:OscIntEstimatesIII} already provides the final estimates in the special case where one
of the exponents $\alpha_1,\alpha_2$ is less than or equal to 0. For the remaining  case an additional
integration by parts argument is needed. This is the main result in this section.

\begin{thm} \label{thm:OscillatoryIntegral}
  Assume $I_\lambda\in\mathcal J_{\lambda,\alpha_1,\alpha_2}$ for $\alpha_1,\alpha_2>-1$ and   $A,B_1,B_2$
  as in \eqref{eq:Psi}. Then we have for all $|\lambda|\geq 1$
  \begin{align*}
     |I_\lambda|\les \rho_1^{\alpha_1+1}\rho_2^{\alpha_2+1}.
  \end{align*}
  The constant depends only on $\alpha_1, \alpha_2$ and on
  the $L^\infty$-norms of finitely many  derivatives of $m,\chi,\Psi$.  
\end{thm}
\begin{proof} 
  We first assume $\alpha_1>0$. Then 
  \begin{align*}
    &\int_{-1}^1 (1-s_1^2)^{\alpha_1} m(s) \chi(\Psi(s)) e^{i\lambda \Psi(s)} \,ds_1 \\
    &\stackrel{\eqref{eq:PsiIntByParts}}= 
    \frac{1}{i\lambda B_1} \int_{-1}^1 (1-s_1^2)^{\alpha_1} m(s)
      \chi_1(\Psi(s))
    \frac{\partial}{\partial s_1} 
    (e^{i\lambda \Psi(s)}) \,ds_1  \\
    &= \frac{i}{\lambda B_1} \int_{-1}^1 
    \frac{\partial}{\partial s_1}  \Big[
    (1-s_1^2)^{\alpha_1} m(s)  \chi_1(\Psi(s)) \Big]    
    e^{i\lambda \Psi(s)} \,ds_1  \\
     &\stackrel{\eqref{eq:properties_Psi}}= 
     \frac{i}{\lambda B_1}  \int_{-1}^1    
    (1-s_1^2)^{\alpha_1-1}   m_1(s)  \chi_1(\Psi(s)) e^{i\lambda \Psi(s)} \,ds_1 \\ 
    &+  \frac{i}{\lambda}   \int_{-1}^1  
    (1-s_1^2)^{\alpha_1} m(s)  \chi_2(\Psi(s))  e^{i\lambda \Psi(s)} \,ds_1.  
  \end{align*}
  This scheme may be repeated   as long   
  as the integrand vanishes on the boundary. In this way, we get for all $M_1\in\{0,\ldots,\lceil
  \alpha_1\rceil\}$
  \begin{align*}
    &\int_{-1}^1 (1-s_1^2)^{\alpha_1} m(s)\chi(\Psi(s)) e^{i\lambda\Psi(s)} \,ds_1 \\
    &= \sum_{\beta_1=0}^{M_1}   
    \left( \frac{i}{\lambda B_1}\right)^{\beta_1} 
    \left( \frac{i}{\lambda}\right)^{M_1 -\beta_1} 
    \int_{-1}^1 (1-s_1^2)^{\alpha_1-\beta_1} m_{\beta_1,M_1}(s) \chi_{\beta_1,M_1}(\Psi(s))  
    e^{i\lambda\Psi(s)} \,ds_1.     
  \end{align*}
  Note that this formula is true also in the case $-1<\alpha_1\leq 0$ where necessarily
  $\beta_1=M_1=\lceil\alpha_1\rceil=0$. So we conclude that the above identity holds under the assumptions
  of the theorem. In the case $\beta_1\in\{0,\ldots,M_1-1\}$  the exponent $\alpha_1-\beta_1$ is positive
  because of $\alpha_1-\beta_1\geq \alpha_1-\lceil \alpha_1\rceil+1>0$. So the integrand vanishes on the
  boundary and we can perform   another integration by parts step. This gives the formula
  \begin{align*}
    &\int_{-1}^1 (1-s_1^2)^{\alpha_1} m(s)\chi(\Psi(s)) e^{i\lambda\Psi(s)} \,ds_1 \\
    &= \left( \frac{i}{\lambda B_1}\right)^{M_1} 
    \int_{-1}^1 (1-s_1^2)^{\alpha_1-M_1} m_{M_1,M_1}(s) \chi_{M_1,M_1}(\Psi(s))  
    e^{i\lambda\Psi(s)} \,ds_1 \\
 	&+ \ind_{M_1\geq 1} \sum_{\beta_1=0}^{M_1}         
    \left( \frac{i}{\lambda B_1}\right)^{\beta_1} 
    \left( \frac{i}{\lambda}\right)^{M_1-\beta_1+1}     
    \int_{-1}^1 (1-s_1^2)^{\alpha_1-\beta_1} \tilde m_{\beta_1,M_1}(s)
    \tilde \chi_{\beta_1,M_1}(\Psi(s)) e^{i\lambda\Psi(s)} \,ds_1. 
  \end{align*}
  This finishes the integration by parts with respect to $s_1$ and we now focus on the integration with
  respect to $s_2$. So we multiply the above expression with $(1-s_2^2)^{\alpha_2}$ and perform an
  analogous analysis for each of these integrals. In this way we obtain, for any given
  $M_1\in\{0,\ldots,\lceil\alpha_1 \rceil\},M_2\in\{0,\ldots,\lceil\alpha_2\rceil\}$, 
  $$
    I_\lambda
    = I_\lambda^1+\ldots+I_\lambda^4
  $$ 
  where
  \begin{align*}
    I_\lambda^1
    &= \left( \frac{i}{\lambda B_1}\right)^{M_1}\left( \frac{i}{\lambda B_2}\right)^{M_2} 
    J_\lambda^1,   \\
    I_\lambda^2
    &=  \ind_{M_2\geq 1} \left( \frac{i}{\lambda B_1}\right)^{M_1}\sum_{\beta_2=0}^{M_2}
      \left( \frac{i}{\lambda B_2}\right)^{\beta_2} 
    \left( \frac{i}{\lambda}\right)^{M_2 -\beta_2+1}  J_{\lambda,\beta_2}^2, \\
    I_\lambda^3
    &= \ind_{M_1\geq 1} \left( \frac{i}{\lambda B_2}\right)^{M_2}\sum_{\beta_1=0}^{M_1}
     \left( \frac{i}{\lambda B_1}\right)^{\beta_1} 
    \left( \frac{i}{\lambda}\right)^{M_1 -\beta_1+1} 
    J_{\lambda,\beta_1}^3, \\
    I_\lambda^4
    &= \ind_{M_1,M_2\geq 1}  \sum_{\beta_1=0}^{M_1}\sum_{\beta_2=0}^{M_2}
    \left( \frac{i}{\lambda B_1}\right)^{\beta_1} 
    \left( \frac{i}{\lambda B_2}\right)^{\beta_2} 
    \left( \frac{i}{\lambda}\right)^{M_1+M_2-\beta_1 -\beta_2+2}
    J_{\lambda,\beta_1,\beta_2}^4
  \end{align*}
  and the integrals belong to the classes
  \begin{align} \label{eq:JClasses}
    \begin{aligned}
    &J_\lambda^1\in \mathcal J_{\lambda,\alpha_1-M_1,\alpha_2-M_2},
    &&\qquad\qquad 
    J_{\lambda,\beta_2}^2\in \mathcal J_{\lambda,\alpha_1-M_1,\alpha_2-\beta_2}, \\
    &J_{\lambda,\beta_1}^3\in\mathcal J_{\lambda,\alpha_1-\beta_1,\alpha_2-M_2},
    &&\qquad\qquad J_{\lambda,\beta_1,\beta_2}^4\in \mathcal
    J_{\lambda,\alpha_1-\beta_1,\alpha_2-\beta_2}.
  \end{aligned}
  \end{align}
     
  \medskip
  
  \textbf{1st case  $|\lambda B_1|,|\lambda B_2|\leq 1$:}\; We choose $M_1=M_2=0$. Then
  $I_\lambda=I_\lambda^1=J_\lambda^1$, so Proposition~\ref{prop:OscIntEstimatesI} and \eqref{eq:JClasses}
  give 
  $$
    |I_\lambda|
    \les |J_\lambda^1| 
    \les 1
    \les \rho_1^{\alpha_1+1}\rho_2^{\alpha_2+1}.
  $$
  
  \medskip
  
  \textbf{2nd case   $|\lambda B_1|\leq 1< |\lambda B_2|$:}\; We choose
  $M_1=0,M_2=\lceil\alpha_2\rceil$. Then  
  $I_\lambda=I_\lambda^1+I_\lambda^2$ with 
  \begin{align*}
    |I_\lambda^1|
    &\les |\lambda B_2|^{-M_2}|J_\lambda^1|
    \les |\lambda B_2|^{-M_2} \rho_1^{\alpha_1+1}\rho_2^{\alpha_2-M_2+1}  
    =  \rho_1^{\alpha_1+1}\rho_2^{\alpha_2+1}. 
    \intertext{Here we used \eqref{eq:JClasses}, Proposition~\ref{prop:OscIntEstimatesIII} and
    $\alpha_2-M_2\in (-1,0]$.
    On the other hand, Proposition~\ref{prop:OscIntEstimatesI} gives} 
    |I_\lambda^2|
    &\les  \sum_{\beta_2=0}^{M_2} 
      |\lambda B_2|^{-\beta_2} |\lambda|^{-M_2+\beta_2-1}
     |J_{\lambda,\beta_2}^2|   \\
    &\les  \sum_{\beta_2=0}^{M_2} 
      |\lambda B_2|^{-\beta_2} |\lambda|^{-M_2+\beta_2-1}
       | B_2|^{-\alpha_2+\beta_2-1}  \\ 
    &\les   \sum_{\beta_2=0}^{M_2}
      |\lambda B_2|^{-\alpha_2-1} |\lambda|^{-M_2+\alpha_2}  
    \;\les\;  \rho_2^{\alpha_2+1}  
    \;\les\;  \rho_1^{\alpha_1+1}  \rho_2^{\alpha_2+1} 
  \end{align*}
  In the last estimate we used $-M_2+\alpha_2\leq 0$.
  
    \medskip
  
  \textbf{3rd case  $|\lambda B_2|\leq 1< |\lambda B_1|$:}\; This is analogous. 
  
  \medskip
  
  \textbf{4th case $|\lambda B_1|,|\lambda B_2|> 1$:}\; We  choose 
  $M_1=\lceil\alpha_1\rceil,M_2=\lceil\alpha_2\rceil$ and obtain from
  Proposition~\ref{prop:OscIntEstimatesIII}  
  $$
    |I_\lambda^1|
    \les |\lambda B_1|^{-M_1} |\lambda B_2|^{-M_2}|J_\lambda^1| 
    \les |\lambda B_1|^{-M_1} |\lambda B_2|^{-M_2} \rho_1^{-\alpha_1+M_1-1} \rho_2^{-\alpha_2+M_1-1}
    = \rho_1^{\alpha_1+1} \rho_2^{\alpha_2+1}.
  $$
   The integral $I_\lambda^2$ is estimated using Proposition~\ref{prop:OscIntEstimatesIII} and
    $-1<\alpha_1-M_1\leq 0$:
  \begin{align*}
    |I_\lambda^2|
    &\les  |\lambda B_1|^{-M_1} \sum_{\beta_2=0}^{M_2}
      |\lambda B_2|^{-\beta_2} |\lambda|^{-M_2+\beta_2-1}
     |J_{\lambda,\beta_2}^2|   \\
    &\les |\lambda B_1|^{-M_1} \sum_{\beta_2=0}^{M_2} 
      |\lambda B_2|^{-\beta_2} |\lambda|^{-M_2+\beta_2-1}
     |\lambda B_1|^{-\alpha_1+M_1-1} | B_2|^{-\alpha_2+\beta_2-1} \\
    &\les   |\lambda B_1|^{-\alpha_1-1} \sum_{\beta_2=0}^{M_2}
      |\lambda B_2|^{-\alpha_2-1} |\lambda|^{-M_2+\alpha_2} \\
    &\les |\lambda B_1|^{-\alpha_1-1} |\lambda B_2|^{-\alpha_2-1}
    \;=\; \rho_1^{\alpha_1+1} \rho_2^{\alpha_2+1}.
   \intertext{
   Similarly one may estimates $I^3_\lambda$ thanks to
   $-M_2+\alpha_2\leq 0$. 
   Finally, Proposition~\ref{prop:OscIntEstimatesI} gives}
    |I_\lambda^4|
    &\les  \sum_{\beta_1=0}^{M_1} \sum_{\beta_2=0}^{M_2}
      |\lambda B_1|^{-\beta_1} |\lambda B_2|^{-\beta_2}
       |\lambda|^{-M_1-M_2+\beta_1+\beta_2-2}
     |J_{\lambda,\beta_1,\beta_2}^4|   \\  
      &\les  \sum_{\beta_1=0}^{M_1} \sum_{\beta_2=0}^{M_2}
      |\lambda B_1|^{-\beta_1} |\lambda B_2|^{-\beta_2}
       |\lambda|^{-M_1-M_2+\beta_1+\beta_2-2}
       | B_1|^{-\alpha_1+\beta_1-1}
       | B_2|^{-\alpha_2+\beta_2-1}  \\
    &\les  \sum_{\beta_1=0}^{M_1} \sum_{\beta_2=0}^{M_2}
      |\lambda B_1|^{-\alpha_1-1} |\lambda B_2|^{-\alpha_2-1}
       |\lambda|^{-M_1-M_2+\alpha_1+\alpha_2}  \\
      &\les |\lambda B_1|^{-\alpha_1-1} |\lambda B_2|^{-\alpha_2-1}
      \;=\; \rho_1^{\alpha_1+1} \rho_2^{\alpha_2+1}.       
  \end{align*}
  So $|I_\lambda|\leq |I_\lambda^1|+\ldots+|I_\lambda^4|$
  gives the claim.
\end{proof}

\medskip

\begin{rem}
  ~ 
  \begin{itemize}
    \item[(a)] In fact stronger bounds can be proved. From the inequalities $\sqrt{A-|B_1|-|B_2|}\leq
    \Psi(s)\leq \sqrt{A+|B_1|+|B_2|}$ and $\supp(\chi)\subset [\frac{1}{2},2]$ we deduce $I_\lambda=0$ whenever 
     $A+|B_1|+|B_2|\leq \frac{1}{4}$ or $A-|B_1|-|B_2|\geq 4$.  
     However, these conditions do 
     not scale nicely in our application, so we omit them.
    \item[(b)] The analysis in this section extends to more general phase functions $\Psi$  satisfying,
    instead of~\eqref{eq:properties_Psi}, $\partial_j \Psi(s)= B_j \varphi_j(\Psi(s))$ where $\varphi_j$ is
    any smooth function on $[\frac{1}{2},2]$. 
  \end{itemize}
\end{rem}

\bibliographystyle{abbrv}
\bibliography{biblio} 
\end{document}